\documentclass{amsart}
\usepackage{parskip,multirow,hyperref,enumitem,soul}
\usepackage{mathtools,amssymb,amsthm}
\usepackage{tikz-cd}

\allowdisplaybreaks

\renewcommand{\AA}{\mathbb{A}}

\newcommand{\CC}{\mathbb{C}}
\newcommand{\F}{\mathbf{F}}
\newcommand{\FF}{\mathbb{F}}
\renewcommand{\H}{\mathbf{H}}

\newcommand{\NN}{\mathbb{N}}

\newcommand{\PP}{\mathbb{P}}

\newcommand{\QQ}{\mathbb{Q}}

\newcommand{\ZZ}{\mathbb{Z}}

\DeclareMathOperator{\Char}{char}
\DeclareMathOperator{\codim}{codim}

\DeclareMathOperator{\Disc}{Disc}

\DeclareMathOperator{\End}{End}
\DeclareMathOperator{\ev}{ev}

\DeclareMathOperator{\im}{im}

\DeclareMathOperator{\PGL}{PGL}
\DeclareMathOperator{\Proj}{Proj}

\DeclareMathOperator{\Rat}{Rat}

\DeclareMathOperator{\Res}{Res}
\DeclareMathOperator{\rk}{rk}
\DeclareMathOperator{\Sep}{Sep}
\DeclareMathOperator{\SepEnd}{SepEnd}

\DeclareMathOperator{\sgn}{sgn}

\DeclareMathOperator{\Spec}{Spec}

\DeclareMathOperator{\Sym}{Sym}

\newcommand{\abs}[1]{\left|#1\right|}

\let\Bar\overline
\let\Tilde\widetilde

\theoremstyle{plain}
\newtheorem{thm}{Theorem}[section]
\newtheorem{lem}[thm]{Lemma}
\newtheorem{cor}[thm]{Corollary}
\newtheorem{propn}[thm]{Proposition}

\theoremstyle{definition}
\newtheorem{defn}[thm]{Definition}
\newtheorem{eg}[thm]{Example}
\newtheorem{q}[thm]{Question}

\theoremstyle{remark}
\newtheorem{rmk}[thm]{Remark}

\newcommand\mc\mathcal
\newcommand\mb\mathbf

\title{Critical loci of self-maps of projective space}

\author{Matt Olechnowicz}
\address{Concordia University}
\email{matt.olechnowicz@concordia.ca}

\author{Max Weinreich}
\address{Harvard University}
\email{mweinreich@math.harvard.edu}

\date{\today}

\begin{document}

\begin{abstract}
Let $K$ be an algebraically closed field and let $n, d \geq 2$. We show that the critical scheme of a general endomorphism of $\PP^n_K$ of degree $d$ is an integral hypersurface. 
This extends a result of Ingram--Ramadas--Silverman to arbitrary characteristic. 
We use basic facts about polynomial rings to show that certain carefully chosen examples have absolutely irreducible Jacobian.
To handle the wild case where the characteristic of $K$ divides $d$, we introduce a polynomial that imitates the homogeneous Jacobian determinant.
\end{abstract}

\maketitle 

\section{Introduction} \label{sec:intro}

The differential of a morphism $f: X \to Y$ between smooth varieties of dimension $r$ over an algebraically closed field $K$ is the induced $K$-vector bundle homomorphism on the Zariski tangent spaces of $X$ and $Y$. After choosing local coordinates on $X$ and $Y$, the differential $f'$ may locally be expressed in matrix form as the Jacobian of $f$. The critical locus of $f$, denoted $C_f$, is defined locally as the vanishing set of $\det f'$. 
The local algebraic expressions for the Jacobian determinant define a closed subscheme structure on $C_f$, which we call the \emph{critical scheme} (Definition \ref{defn_cf_fulton_style}).

Our goal in this paper is to describe the critical schemes of general maps of projective spaces. 
To do this, we work in the parameter space $\End^n_{d/K}$ of endomorphisms of $\PP^n_K$ of algebraic degree $d$ introduced by Silverman \cite{silvermanGIT}. 
Each endomorphism $f: \PP^n \to \PP^n$ of algebraic degree $d$ may be represented in homogeneous coordinates by
\[[x_0 : \ldots : x_n] \mapsto [F_0(x_0, \ldots, x_n) : \ldots : F_n(x_0, \ldots, x_n)],\]
where $(F_0, \ldots, F_n)$ is a tuple of homogeneous forms of degree $d$ over $K$ that admit no common zero. The representation of $f$ by a tuple of degree $d$ is unique up to simultaneous scaling of the forms $F_0, \ldots, F_n$.  One views $\End_{d/K}^n$ as a quasi-projective variety by identifying $f = [F_0 : \ldots : F_n]$ with the vector of coefficients of the monomials in $F_0, \ldots, F_n$, up to scale.

\begin{thm} \label{main}
Let $K$ be an algebraically closed field. 
For all $n, d \ge 2$ 
the set
\[
\mathcal{E}_{n,d} := \{f \in \End^n_{d/K} : \text{the critical scheme $C_f$ is an integral hypersurface}\} 
\]
is non-empty and Zariski-open.
\end{thm}
The word \emph{integral} is here used in the sense of scheme theory; a hypersurface is integral if and only if it is irreducible and reduced, i.e.\ irreducible of multiplicity $1$.


We interpret Theorem \ref{main} as a rough description of the critical locus of a general endomorphism. Indeed, Zariski-open subsets of varieties are dense as soon as they are nonempty, so Theorem \ref{main} describes ``almost all'' endomorphisms. 

It is unsurprising that the set $\mc{E}_{n,d}$ is Zariski-open. The difficulty lies in demonstrating nonemptiness, since easily computed critical schemes tend to be degenerate. As an extreme case, consider the power map of degree $d \geq 2$, defined by
$$f: \PP^n \to \PP^n,$$
$$[x_0 : \ldots : x_n] \mapsto [x_0^d : \ldots x_n^d].$$
Its critical scheme $C_f$ is 
the union of the $n+1$ coordinate hypersurfaces $\{x_i = 0\}$, each counted with multiplicity $(d - 1)$---unless $\Char K \mid d$, in which case $C_f = \PP^n$. In particular we observe that $C_f$ may have multiple components, may be nonreduced, or may fail to be a hypersurface. (For even worse, see Example \ref{eg_most_pathological}.) The upshot of Theorem \ref{main} is that all these features of the power map are exceptional.

Following an idea of Starr, Ingram--Ramadas--Silverman proved Theorem \ref{main} in the case $\Char K = 0$ except when both $d=2$ and $n\geq 3$ \cite[Theorems 14 and 15(a)]{IRS}. Their proof uses the characteristic-zero assumption in an essential way and depends on some results from the theories of surfaces and determinantal varieties. 
Consequently, it is unclear how to extract any example of an element of $\mc{E}_{n,d}$ from their method. 

The motivation for the present work was to offer a more direct proof by identifying, for each pair $(n,d)$, a concrete example of a degree-$d$ endomorphism $f: \PP^n \to \PP^n$ for which $C_f$ is an integral hypersurface. We obtain the following family of examples.

\begin{thm} \label{morphism example intro}
Let $n, d \ge 2$ and let $p$ be an odd prime not dividing $d-1$.
Define
\begin{align*}
    F_0 &:= px_0^d + 2x_0 x_1 \frac{x_1^{d-1}-x_0^{d-1}}{x_1-x_0} \\ 
    F_i &:= px_i^d + 2x_1^{d-1} x_i - 2x_0^{d-1} x_{i+1} \qquad (1 \le i \le n)
\end{align*}
where $x_{n+1} = x_0$.
Then 
\[f := [F_0 : \ldots : F_n]\]
is a degree-$d$ endomorphism of $\PP^n$ defined over $\QQ$ whose critical scheme $C_f$ is an integral hypersurface.
\end{thm}

\subsection{Critical theory}
We now discuss the new methods that we introduce to study critical schemes over general algebraically closed fields. We refer to the situation where $\Char K \nmid d$ as the \emph{tame} setting, and $\Char K \mid d$ as the \emph{wild} setting. While the tame setting largely matches characteristic $0$, the wild setting poses many interesting problems. (We prove most of our results for rational maps rather than just endomorphisms, but for ease of exposition we only discuss endomorphisms in the Introduction.)

The definition of the critical scheme $C_f$ is local, but for self-maps of $\PP^n$ over $\CC$, one often sees a global definition of critical locus in terms of the vanishing of a single Jacobian determinant.
In characteristic $0$, the critical scheme $C_f$ of an endomorphism $f = [F] = [F_0 : \ldots : F_n]$ coincides with the closed subscheme $V_{\PP^n}(J_F)$ cut out by the \emph{homogeneous} Jacobian determinant
\[
J_{F} := \det \frac{\partial(F_0, \ldots, F_n)}{\partial(x_0, \ldots, x_n)}
=
\begin{vmatrix}
\dfrac{\partial F_0}{\partial x_0} & \ldots & \dfrac{\partial F_0}{\partial x_n} \\
\vdots & \ddots & \vdots \\
\dfrac{\partial F_n}{\partial x_0} & \ldots & \dfrac{\partial F_n}{\partial x_n}
\end{vmatrix}.
\]
From the algebraic point of view, there are two difficulties in working with $V_{\PP^n}(J_{F})$. 
First, the global definition by $V_{\PP^n}(J_F)$ fails to capture any information in the wild setting $\Char K \mid d$, because then $J_{F}$ is identically $0$; see Proposition \ref{Cf is VJF}. Second, while it is relatively clear that the critical scheme $C_f$ and $V_{\PP^n}(J_{F})$ agree as sets in the tame setting $\Char K \nmid d$ (Proposition \ref{Cf is VJF}), it is harder to show that their scheme structures coincide.

To remedy these issues, we introduce a polynomial, the \emph{flat Jacobian $J^*_F$}, that extends the useful properties of $J_F$ to all maps, tame or wild. The inspiration for the definition is the observation that $J_F$ is always divisible by the degree $d$ of the homogeneous polynomials $F_i$ making up $F$. We study the generic homogenous Jacobian determinant $J_{n,d}$ over $\ZZ$ (Definition \ref{def_Jnd}) and observe that $d \mid J_{n,d}$. To correct for this, when $d \geq 1$, we define the generic \emph{flat} Jacobian over $\ZZ$ by
$$J^*_{n,d} := \frac{1}{d} J_{n,d}.$$
Then $J^*_F$ is defined by specializing $J^*_{n,d}$ to the base field $K$ and substituting in the coefficients of $F$. Our key result regarding the flat Jacobian $J^*_F$ is that it correctly describes the scheme structure of the critical locus.
\begin{thm} \label{thm_intro_Jstar}
Let $n, d \ge 1$. 
Given
$f = [F] = [F_0 : \ldots : F_n] \in \End_{d/K}^n$, let $J^*_F$ denote its flat Jacobian. Then $V_{\PP^n}(J^*_F) = C_f$ as schemes.
\end{thm}
In characteristic $0$, there is no difference between the vanishing schemes of $J_F$ and $J^*_F$, so this result justifies the folklore equivalence $C_f = V_{\PP^n}(J_F)$. In characteristic $p$, we obtain another interpretation, which should be useful in the study of inseparability.
\begin{cor} \label{cor_intro_insep}
Let $n, d \geq 1$. Then for each
$f = [F] = [F_0 : \ldots : F_n] \in \End_{d/K}^n$, the following are equivalent:
\begin{enumerate}
    \item the map $f$ is inseparable;
    \item $J^*_F = 0$;
    \item $C_f = \PP^n$.
\end{enumerate}
\end{cor}
We use this to show that the parameter space $\SepEnd_{d/K}^n$ of separable endomorphisms within $\End_{d/K}^n$ is a nonempty-Zariski open set, and there is an open subscheme $\SepEnd_{d/\ZZ}^n$ of Silverman's parameter space of endomorphisms over $\ZZ$ that parametrizes separable endomorphisms.

The terminology ``flat'' for the flat Jacobian $J^*_{n,d}$ is informal, and motivated as follows. Since $J_{n,d}$ is a bihomogeneous polynomial over $\ZZ$, it defines a family of subschemes of $\PP^n_\ZZ$ over $\End^n_{d/\ZZ}$. This family is very far from being flat; its general fiber is a hypersurface, but its fibers at primes $p \mid d$ jump in dimension. The ``flat'' Jacobian $J^*_{n,d}$ avoids these bad fibers, and defines a flat family of hypersurfaces over the larger base $\SepEnd_{d/\ZZ}^n$. For details, see Section \ref{sect_separable_scheme}.

\subsection{Proof structure}
Theorem \ref{thm_intro_Jstar}, which relates the flat Jacobian to the critical scheme, is proved by a delicate computation in affine charts which shows the polynomials defining each agree. With Theorem \ref{thm_intro_Jstar} in hand, Theorem \ref{main} boils down to the existence (for any given field $K$ and pair $n, d \geq 2$) of a tuple $F = (F_0, \ldots, F_n) \in K[x_0, \ldots, x_n]^{n+1}$ of homogeneous polynomials of degree $d$ such that $J^*_F$ is irreducible in $K[x_0, \ldots, x_n]$.
See Lemma \ref{lem_specific_implies_general} for the proof of this reduction step.

This brings us to the most technical part of the article: the search 
for an irreducible example.
Locating a reasonable candidate $F$ is difficult because $J^*_F$ has high degree in $x_0, \ldots, x_n$, and its coefficients are complicated polynomials in the coefficients of $F$.
Through much trial and error, we discovered a family of tuples $F$ for which the homogeneous Jacobian matrix $D_F$ is supported in essentially just two columns and two diagonals (see Proposition \ref{unified eg}).
Proving irreducibility is hard and requires casework depending on characteristic. 
The tricks we use all belong to elementary abstract algebra, but they are rather \textit{ad hoc}, and we defer further discussion to Section \ref{sec:eg motivation}. This proves Theorem \ref{main}.

The exemplary tuple $F$ just described has common zeros among its components, so it is not itself an element of $\End_{d/K}^n$. Rather, it defines a rational map. Thanks to the rigidity of the Zariski topology, a rational map is enough to prove Theorem \ref{main}. But one really wants to see an honest-to-goodness morphism, and ruling out common zeros is usually a cumbersome resultant calculation. Instead, we use an interpolation trick to manufacture the morphism-example of Theorem \ref{morphism example intro}. The method is to construct a tuple $F$ over $\QQ$ that is a morphism modulo $2$ and has integral critical scheme modulo a well-chosen prime $p$. This allows us to check the two conditions independently, preventing interference between the two calculations. Even though Theorem \ref{morphism example intro} is set in characteristic $0$, which is generally easier, characteristic-$p$ methods are an essential ingredient in the proof.

\subsection{What else is known}

\subsubsection{Critical loci in algebraic geometry and dynamics}
Critical loci are ubiquitous in calculus and differential geometry. In algebraic geometry, the critical scheme is a fundamental tool due to its connections to ramification, the Riemann--Hurwitz formula, the \'etale property, and separability. Indeed, for separable dominant maps between varieties of the same dimension, the critical scheme $C_f$ may be identified with the Weil divisor
\[\sum_{\textrm{prime divisors } V} \mu_V [V],\]
where the multiplicity $\mu_V$ is the length of the relative canonical sheaf at $V$.
Most sources call this the \emph{ramification divisor}, because, in characteristic $0$, for each hypersurface $V \subset X$ such that $f(V)$ is a hypersurface, the extension of discrete valuation rings induced by $f$ at $V$ has ramification index $\mu_V + 1$. In particular, in characteristic $0$, the critical locus of a dominant endomorphism $f: \PP^n \to \PP^n$ is also the support of the set of ramified hypersurfaces, providing an alternate divisor structure of interest on $C_f$. In positive characteristic, the relationship between the critical scheme and DVR ramification indices is more subtle due to the possibility of wild ramification, so we distinguish between these structures (and focus solely on the former).

Critical loci are also of fundamental importance in algebraic and complex dynamics. Over $\CC$, the two definitions of the critical locus seem to be used interchangeably. For uses of the local definition, see e.g.\ \cite[Definition 2.1]{FS92}, 
\cite[\S 4.2]{Ueda}, 
\cite[\S 2]{Rong}, 
\cite[\S 1]{Astorg}, 
\cite[before Definition 2]{GHK}, 
\cite[p.\ 4]{GTV}. 
For uses of the global definition, 
see e.g.\ \cite[Definition 1]{IRS}, 
\cite[\S 1]{Ole}, 
\cite[Proposition 5.8]{Koch} 
\cite[Corollary 1]{Astorg}, 
\cite[Theorem 2.4]{FS94}, 
\cite[\S 1.2]{Sibony1999}, 
\cite[Proposition 1.6]{DS}. 

The \emph{branch locus} of a map is the image of its critical locus. Other works are devoted to understanding the geometry of branch loci, e.g.\ \cite{Gurjar}. In particular, post-critically finite maps (i.e.\ endomorphisms for which the branch locus is contained in the critical locus) are an active research area, cf.\ \cite{Ole, IRS, GTV, Astorg}. While the present article contains no dynamics \textit{per se}, it was our interest in post-critically finite maps that originally led us to the problem at hand. 
In particular, Theorem \ref{main} lets us remove the restriction on $n$ and $d$ in \cite[Main Theorem]{Ole}.
Thus, we now know that post-critically dynamically improper maps are Zariski-dense in $\End^n_d$ for all $n, d \ge 2$ (at least in characteristic $0$).

\subsubsection{Necessity of the hypotheses}
In Theorem \ref{main}, we assume $n, d \geq 2$. These assumptions are necessary.

When $d = 1$, endomorphisms of $\PP^n$ are invertible, so their critical schemes are empty, hence non-irreducible.

When $n = 1$, the critical scheme $C_f$ is either all of $\PP^1$ (the inseparable case) or is a multiset of degree $2d-2$, and the latter are not integral. If further $\Char K = 2$, then due to wild ramification, the multiplicity of each point in the critical scheme is at least $2$, so the critical scheme is not even reduced!

Neither may we replace $\PP^n$ in Theorem \ref{main} by an arbitrary variety of dimension $n > 1$. 
For instance, every surjective endomorphism $f$ of $\PP^1 \times \PP^1$ is of the form \[f(P, Q) = (a(P), b(Q) )\]
for some nonconstant morphisms $a, b : \PP^1 \to \PP^1$, up to a permutation of components. If $\deg a$ and $\deg b$ are at least $2$, then $C_f$ contains both horizontal and vertical components and is therefore reducible.

\subsubsection{Comparison to Ingram--Ramadas--Silverman}

Considerably more is known regarding the geometry of critical loci in characteristic $0$, especially due to the work of Ingram--Ramadas--Silverman \cite{IRS}.
Since they work with the homogeneous Jacobian determinant $J_F$, their methods are not directly applicable in the wild setting.
 
Our work was inspired by \cite[Theorem 14]{IRS}: for all $n \geq 2$ and $d \geq 3$, the critical hypersurface is irreducible. (In fact, the proof shows integrality, which is stronger.) We thank the authors for explaining this argument to us in detail. Their technique uses an incidence correspondence to estimate the codimension of the singular locus of the scheme $C_f$, showing that for a general map $f$, this codimension is at least $3$. But since non-integral projective hypersurfaces have singular locus of codimension at most $1$, it follows that these maps have integral critical scheme.

That result is complemented by \cite[Theorem 15(a)]{IRS}: for all $d \geq 2$, the critical scheme of a general degree-$d$ endomorphism of $\PP^2$ is smooth and irreducible. (The theorem was stated for $d \geq 1$, but 
    our convention in this work is that 
the empty scheme is not integral.)
The proof relies on surface singularity theory \cite{Kulikov}.

Applying integrality, Ingram--Ramadas--Silverman go further and show that if $n \geq 2$ and $d \geq 3$, the critical scheme of a general map is of general type. This argument, suggested by Jason Starr, uses more advanced techniques (desingularizations of determinantal varieties, canonical singularity theory) due to Vainsencher \cite{Vainsencher} and Starr \cite{Starr}.

\begin{q}
For which pairs $(n,d)$ is the critical locus of a general degree-$d$ endomorphism of $\PP^n$ smooth?
\end{q}

\begin{q}
Does there exist a morphism $f$ such that $C_f$ is rational?
\end{q}

\subsection{Outline of paper}

Section \ref{sec:prelim} contains preliminaries from elementary abstract algebra, as well as facts about spaces of polynomials and self-maps of projective spaces.

Section \ref{sec:CJ} formally defines the critical scheme and establishes its relationship to the homogeneous Jacobian determinant, leading to the division between the tame and wild settings.

Section \ref{sec:J*} introduces the flat Jacobian $J^*_{n,d}$. Theorem \ref{thm_intro_Jstar} is proved as a special case of Theorem \ref{thm_flat_J_poly}, which shows that the flat Jacobian defines the critical scheme of a rational self-map of projective space.

Section \ref{sec:eg} reduces Theorem \ref{main} to finding a tuple $F$ with irreducible flat Jacobian $J^*_F$, then constructs the desired tuple (thereby proving said Theorem).

Section \ref{sec:morphisms} describes our method for promoting rational maps with integral critical scheme to morphisms. We prove Theorem \ref{morphism example intro} in the text as Theorem \ref{morphism example}.

Section \ref{sec:insep} relates the flat Jacobian $J^*_{n,d}$ to inseparability, proving Corollary \ref{cor_intro_insep}. This section may be read independently from Sections \ref{sec:eg} and \ref{sec:morphisms}.

\section*{Acknowledgments}
We are grateful to Joseph Silverman and Patrick Ingram for their patient explanation of the methods in \cite{IRS}. We also thank Richard Birkett, Nathan Chen, Laura DeMarco, Anna Dietrich, Sarah Koch, Aaron Landesman, John Lesieutre, Curt McMullen, Yohsuke Matsuzawa, and Shuyang Shen for helpful conversations at various stages of this project. Max Weinreich was supported by NSF Grant 2202752.

\section{Preliminaries} \label{sec:prelim}

All algebro-geometric definitions in this paper are consistent with the textbook of Hartshorne \cite{Hartshorne}. In particular, a \emph{variety} is an integral, separated scheme over an algebraically closed field. An important edge case is that the empty scheme $\varnothing$ is reduced, but not irreducible (the nilradical of the zero ring is not prime); hence, the empty scheme is not integral, and not a variety.
Beware that $0 \in K[\bf x]$ is not irreducible as a ring element, whereas $V(0) = \PP^n$ \emph{is} irreducible as a scheme.
\subsection{Facts about polynomials}

In this section, we adopt the following notations:
\begin{center}
\begin{tabular}{cl}
    $R$ & commutative unital ring \\
    $R[z]$ & univariate polynomial ring over $R$ \\
    $R[\mathbf z]$ & multivariate polynomial ring over $R$
\end{tabular}
\end{center}

We recall some well-known facts about polynomial rings that appear in our proofs.

\begin{lem}[Euler's theorem] \label{euler}
Let $\Phi \in R[\mathbf z]$ be a homogeneous polynomial (i.e.\ a \emph{form}).
Then: 
\[\sum_i z_i \frac{\partial \Phi}{\partial z_i} = (\deg \Phi) \Phi.\]
\end{lem}

\begin{proof}
Both sides are linear in $\Phi$ of fixed (total) degree, and the claim is immediate for monomials.
\end{proof}

Recall that an element $r$ of an integral domain $R$ is said to be \emph{irreducible} if and only if the following all hold: 
(i) $r \ne 0$, 
(ii) $r \not \in R^\times$, 
(iii) $r = ab$ implies $a$ or $b$ is a unit.

\begin{lem}[a geometric irreducibility criterion] \label{geo irr}
Let $R$ be an integral domain
and let $\ev : R[\mathbf z, \mathbf w] \to R[\mathbf z]$ be the evaluation map sending each $\bf w$ to $0$.
Suppose $\Phi \in R[\mathbf z, \mathbf w]$ is homogeneous (in all the variables).
If $\ev(\Phi)$ is irreducible, then $\Phi$ is irreducible.
\end{lem}

\begin{proof}
The evaluation of a factorization is also a factorization of the evaluation.
\end{proof}

\begin{lem}[a linear irreducibility criterion] \label{lin irr}
Let $R$ be an integral domain and let $a, b \in R$ with $a \ne 0$.
Then $az + b \in R[z]$ is irreducible
if and only if $a$ and $b$ are coprime, i.e.\ the only principal ideal containing $a$ and $b$ is the unit ideal.
\end{lem}




\begin{proof}
Suppose $a$ and $b$ are coprime. 
Write $az + b = pq$ for some $p, q$ in $R[z]$ with $\deg p \leq \deg q$, say.
Since $a \ne 0$, we have $\deg (az + b) = 1$. 
Since $R$ is an integral domain, $\deg (pq) = \deg p + \deg q$. 
Thus $\deg p + \deg q = 1$, 
and it follows that $\deg p = 0$ and $q = rz + s$ for some $r, s$ in $R$.
Then $az + b = prz + ps$, so comparing coefficients shows $a = pr$ and $b = ps$. 
Thus $p$ divides $a$ and $b$ in $R$, so by coprimality $p$ divides $1$, i.e.\ $p$ is a unit.

Conversely, suppose $az + b$ is irreducible.
Let $d \in R$ be a common divisor of $a$ and $b$, say $a = md$ and $b = nd$.
Then $az + b = (mz + n)d$, so one of $mz + n$ or $d$ is a unit. 
But since $R$ is an integral domain, $R[z]^\times = R^\times$; 
and since $a \ne 0$, we have $m \ne 0$. 
Thus $mz + n$ is not a unit, so $d$ is a unit.
\end{proof}

\subsubsection{Projective roots}

In this subsubsection, we prove an elementary but technical lemma that is used only in Section \ref{sec:eg}.

Let $K$ be an algebraically closed field. A \emph{projective root} of a binary form $G \in K[x, y]$ is a point of $V(G)(K) \subseteq \PP^1$, 
i.e.\ a point $[\alpha : \beta]$ such that $(\alpha, \beta) \in K^2 \smallsetminus \{(0, 0)\}$ and $G(\alpha, \beta) = 0$. 
If $[\alpha : \beta]$ is a projective root of $G$, 
then $G$ is divisible by the linear form $\beta x - \alpha y$. 
When $G \ne 0$ is divisible by $(\beta x - \alpha y)^2$ we say $[\alpha : \beta]$ is a \emph{repeated root}; 
in this case the scheme $V(G)$ is non-reduced (equivalently, singular) at $[\alpha:\beta]$.

Recall that the \emph{resultant} of two homogeneous polynomials 
\[G = u_p x^p + \ldots + u_0 y^p \quad\text{and}\quad H = v_q x^q + \ldots + v_0 y^q\]
is the determinant $\Res_{p,q}(G, H)$ of the $(q+p)$-by-$(q+p)$ Sylvester matrix 
\[
\begin{bmatrix}
  u_p  &        &        &   v_q  &        &        \\
\vdots & \ddots &        & \vdots & \ddots &        \\
\vdots &        &   u_p  & \vdots &        &   v_q  \\
  u_0  &        & \vdots &   v_0  &        & \vdots \\
       & \ddots & \vdots &        & \ddots & \vdots \\
       &        &   u_0  &        &        &   v_0  
\end{bmatrix}
\]
where there are $q$ columns of $u$'s and $p$ columns of $v$'s \cite[p.~77]{Cox}.
If $G$ and $H$ are defined over $K$, 
then $\Res(G, H) = 0$ if and only if $G$ and $H$ have a common projective root over $K$.

The resultant is a generalization of the determinant to nonlinear forms. 
It has a number of useful properties that assist in its calculation, including:

\begin{itemize}
\item (\textit{multiplicativity}) $\Res(A, BC) = \Res(A, B) \Res(A, C)$;
\item (\textit{homogeneity}) $\Res(aA, bB) = a^{\deg B} b^{\deg A} \Res(A, B)$; 
\item (\textit{congruence}) $\Res(A, B) = \Res(A, C)$ whenever $B \equiv C$ (mod $A$) and $\deg B = \deg C$.
\end{itemize}

Recall that the \emph{discriminant} of a homogeneous polynomial 
\[G(x, y) = u_p x^p + u_{p-1} x^{p-1} y + \ldots + u_0 y^p\]
is a certain polynomial $\Disc(G)$ in the coefficients of $G$ that vanishes in $K$ if and only if $G$ has a repeated root over $K$.
For generic $G$, the discriminant may be defined by
\[
\Disc(G) := \frac{\Res(G, \partial_x G)}{\Res(G, y)} \in \ZZ[{\bf u}].
\]
Beware that the sign differs from the usual convention, e.g.\ $\Disc(ax^2 + bxy + cy^2) = 4ac - b^2$. 
See \cite[3.1.4]{BJ} for details.

\begin{lem} \label{res jac disc}
Let $R$ be a (commutative, unital) ring 
and let $G, H \in R[x, y]$ be nonzero homogeneous polynomials (not necessarily of the same degree).
Then 
\[\Res\bigg(\! G, \begin{vmatrix} \partial_x G & \partial_y G \\ 
\partial_x H & \partial_y H \end{vmatrix}\bigg)
= (\deg H)^{\deg G} \Disc(G) \Res(G, H).
\]
In particular, 
\[
\Res(G, (\partial_x G) y^k + (\partial_y G) x^k) = \Disc(G) \Res(G, y^{k+1} - x^{k+1})
\]
for all $k \ge 0$.
\end{lem}

\begin{proof} 
By specialization, it suffices to work over $R = \ZZ[{\bf u}, {\bf v}]$ with
\[G = u_p x^p + \ldots + u_0 y^p \quad\text{and}\quad H = v_q x^q + \ldots + v_0 y^q\]
universal forms.
For brevity write $\partial_x G$ as $G_x$, etc.
Then 
\[J := \begin{vmatrix} G_x & G_y \\ H_x & H_y \end{vmatrix} = G_x H_y - G_y H_x.\]
By Euler's theorem \ref{euler} applied to $G$ and then $H$,
\begin{align*}
yJ &= yG_x H_y - y G_y H_x \\
&\equiv yG_x H_y + xG_x H_x \pmod G \\
&= G_x (yH_y + xH_x) \\ 
&= q G_x H.
\end{align*}
Moreover,
\[\deg {(yJ)} = 1 + (p-1) + (q-1) = (p-1) + q = \deg {(q G_x H)}\]
are equal.
Therefore 
\begin{align*}
\Res(G, y) \Res(G, J)
&= \Res(G, yJ) & \text{by multiplicativity} \\
&= \Res(G, q G_x H) & \text{by congruence} \\
&= q^p \Res(G, G_x H) & \text{by homogeneity} \\
&= q^p \Res(G, G_x) \Res(G, H) & \text{by multiplicativity} \\
&= q^p \Disc(G) \Res(G, y) \Res(G, H) & \text{by definition.}
\end{align*}
Now $\Res(G, y) = u_p$ is the leading $x$-coefficient of $G$;
since it is nonzero, we may cancel it and obtain 
\[
\Res(G, J) = q^p \Disc(G) \Res(G, H) \tag{*}
\]
in $\ZZ[{\bf u}, {\bf v}]$. This proves the first part. 
For the second, apply (*) with $H := y^{k+1} - x^{k+1}$ and cancel the emergent common factor $(k+1)^p$.
\end{proof}


\subsection{Spaces of forms} \label{sec:prelim forms}

In this section, we adopt the following notation:
\begin{center}
\begin{tabular}{cl}
    $n, d$ & non-negative integers \\
    $K$ & an algebraically closed field. \\
\end{tabular}
\end{center}

Given a $K$-vector space $V$, let $\Sym(V)$ be the symmetric algebra on $V$ over $K$. Then $\Sym(V)$ is a polynomial ring, and the number of indeterminates is $\dim V$.
 
The projectivization of $V$ is the variety
\[\PP(V) := \Proj(\Sym(V)) \cong \PP^{\dim V - 1}_K.\]
Concretely, the set of $K$-points of $\PP(V)$ 
is the set $(V \smallsetminus \{0\}) / K^\times$ of equivalence classes $[v]$ of nonzero vectors modulo nonzero scalars.

Recall that a \emph{form} is a homogeneous polynomial.
Let $K[{\bf x}]_d$ denote the vector subspace of $K[{\bf x}]$ consisting of forms of degree $d$ (including $0$).
Let 
$$\F^n_d := \PP(K[\mathbf{x}]_d),$$
$$\H^n_d := \left\{[\Phi] \in \F^n_d(K) : \Phi \textrm{ is irreducible} \right\}.$$

\begin{propn} \label{Hnd nonempty}
The set $\H^n_d$ is Zariski open, and is nonempty if and only if $d = 1$ or if $n \geq 2$ and $d \geq 2$.
\end{propn}
\begin{proof}
First we handle the trivial cases.
If $d=0$, then $\mathbf{F}^n_0 = \{[1]\}$; but since irreducibles are non-units, $\H^n_0 = \varnothing$.
If $d=1$, then $\H^n_1 = \F^n_1 \neq \varnothing$ because all linear forms are irreducible.
So, assume $d \ge 2$.
For any partition $d_1 + \ldots + d_r = d$, consider the multiplication map
\begin{gather*}
    \mu : \F^n_{d_1} \times \ldots \times \F^n_{d_r} \to \F^n_d \\
    ([\Phi_1], \ldots, [\Phi_r]) \mapsto [\Phi_1 \ldots \Phi_r].
\end{gather*}
Then
\[\H^n_d = \F^n_d \smallsetminus \bigcup \, {\im {\mu}}\]
where the union runs over all $\mu$ associated to partitions with $r \geq 2$. 
Thus $\H^n_d$ is Zariski open. 
Now we check when $\H^n_d$ is empty, by cases. 
If $n \leq 1$ then $\H^n_d = \varnothing$ because all binary forms split completely ($K$ is algebraically closed); 
whereas if $n \geq 2$ then $\H^n_d$ contains $[x_0^d + x_1^{d-1} x_2]$ by Lemma \ref{lin irr}.
\end{proof}

\subsection{Spaces of maps} \label{sec:prelim maps}

In this subsection, we describe parameter spaces of self-maps of degree $d$ of projective space $\PP^n$ over
an algebraically closed field $K$. To avoid trivialities, assume $n \ge 1$. 

A tuple $(F_0, \ldots, F_m)$ in $K[\mathbf{x}]^{m+1}$ consisting of forms all of the same degree and not all $0$ defines a rational map
\[f : \PP^n \dashrightarrow \PP^m\]
\[[x_0 : \ldots : x_n] \mapsto [F_0(x_0, \ldots, x_n) : \ldots : F_m(x_0, \ldots, x_n)].\]
Every rational map $f : \PP^n \dashrightarrow \PP^m$ may be represented this way, and the defining tuple is unique up to multiplication by elements of $K[\mathbf{x}] \smallsetminus \{0\}$.
If $f$ is representable by a tuple $F$ of forms of degree $d$ with no common factor, we say that $f$ has \emph{degree $d$}, and we call $F$ a \emph{lift} of $f$.

A rational map $f : X \dashrightarrow Y$ of varieties has a maximal domain of definition $U_f$, which is a nonempty open subset of $X$. The indeterminacy locus of $f$ is the closed proper subvariety $I_f := X \smallsetminus U_f$. A rational map $f$ is a morphism if and only if $I_f = \varnothing$.

In the case of a map $f : \PP^n \dashrightarrow \PP^m$, we may view $I_f$ as the closed subscheme $V_{\PP^n}(F_0, \ldots, F_m)$, where $(F_0, \ldots, F_m)$ is any lift of $f$.
Further, in this situation $\dim I_f \leq n - 2$ because components of codimension 1 yield common factors.

We now define spaces of rational self-maps and endomorphisms. We construct these as open subsets of the projective space
\[\Bar{\End}^n_d := \PP(K[\mathbf{x}]_d^{n+1}).\]
Let 
\[\Rat^n_d := \{[F] \in \Bar\End^n_d : [F] \text{ defines a rational map } \PP^n \dashrightarrow \PP^n \text{ of degree } d\}\]
and 
\[\End^n_d := \{[F] \in \Bar\End^n_d : [F] \text{ defines a morphism } \PP^n \to \PP^n \text{ of degree } d \}. \]

The next lemma shows how to view these spaces as varieties.

\begin{lem} \label{rat end qp}
For all $n \ge 1$ and $d \ge 0$, the sets
\[\End^n_d \subseteq \Rat^n_d \subseteq \Bar{\End}^n_d\]
are nonempty and Zariski-open.
\end{lem}
\begin{proof}
To see that $\Rat_d^n$ is a Zariski-open subset of $\Bar{\End}_d^n$, observe that
$$\Rat_d^n = \{[F] \in \Bar{\End}^n_d : \gcd(F_0, \ldots, F_n) = (1)\}.$$
Note that the gcd condition is well-defined on projective equivalence classes. 
Now, the complement of $\Rat_d^n$ is the union over $0 < e < d$ of the images of the multiplication maps
\begin{gather*} 
    \F^n_e \times \Bar\End^n_{d-e} \to \Bar\End^n_d \\
    ([\Phi], [G]) \mapsto [\Phi G_0 : \ldots : \Phi G_n].
\end{gather*}
So $\Rat_d^n$ is open. The set $\End_d^n$ is open as well, because it is the complement of the vanishing locus of the Macaulay resultant of $F_0, \ldots, F_n$.
For details, see \cite[Section 2]{Levy} or \cite[Chapter 3]{Cox}. Since $\End_d^n$ contains $[x_0^d : \ldots : x_n^d]$, it is nonempty.
\end{proof}

\begin{eg}[$d=1$]
Let $\mathsf{M}_{n+1}$ denote the space of $(n+1)$-by-$(n+1)$ square matrices. 
We may identify $\Bar{\End}^n_1$ with $\PP(\mathsf{M}_{n+1})$ via the map 
\[
F = (a_{00} x_0 + \ldots + a_{0n} x_n, \ldots, a_{n0} x_0 + \ldots + a_{nn} x_n)
\mapsto 
D_F := \begin{bmatrix}a_{00} & \ldots & a_{0n} \\ \vdots & \ddots & \vdots \\ a_{n0} & \ldots & a_{nn} \end{bmatrix}
\] 
on affine lifts.
Under this identification, we have
\begin{align*}
    \Rat^n_1 &= \{[F] : \rk D_F > 1\}
    \shortintertext{and} 
    \End^n_1 &= \{[F] : \det D_F \ne 0\} \cong \PGL_{n+1}.
\end{align*}
To wit, the linear forms $F_j = a_{j0} x_0 + \ldots + a_{jn} x_n$ are \emph{not} coprime if and only if they are all scalar multiples of a single form $G = b_0 x_0 + \ldots + b_n x_n$. In this case, $F_j = a_j G$ for some scalars $a_j$ (not all zero), and the matrix 
\[D_F 
= \begin{bmatrix}a_{00} & \ldots & a_{0n} \\ \vdots & \ddots & \vdots \\ a_{n0} & \ldots & a_{nn} \end{bmatrix} = 
\begin{bmatrix} a_0 b_0 & \ldots & a_0 b_n \\ \vdots & \ddots & \vdots \\ a_n b_0 & \ldots & a_n b_n\end{bmatrix} 
= 
\begin{bmatrix} 
    a_0 \\
    \vdots \\
    a_n
\end{bmatrix} 
\begin{bmatrix} b_0 & \ldots & b_n \end{bmatrix} \]
has rank 1.
More generally, 
\begin{align*}
    \Rat^n_d &= \{[F] : \codim I_F > 1\}
\end{align*}
where $I_F := V_{\PP^n}(F_0, \ldots, F_n)$.
\end{eg}

\section{Critical loci and Jacobians} \label{sec:CJ}

\subsection{The critical scheme of a rational map}

In this section, we define the critical locus $C_f$ of a morphism $f: X \to Y$ of smooth varieties of the same dimension, and then we recall the textbook definition of its scheme structure. 
These constructions appear throughout algebraic geometry, especially in the study of ramification; see e.g.\ \cite[Example 3.2.20]{fulton}.

Intuitively, one would like to interpret $C_f$ as the closed subscheme of $X$ defined by the determinant of the differential $f'$ of $f$. However, some care is required in the definition since the familiar differential $TX \to TY$ between tangent bundles from manifold theory is not a vector bundle homomorphism in the sense of algebraic geometry---the domain and codomain are vector bundles over different base spaces.
For background on the necessary vector bundle operations, see \cite[Chapter 14, pp.\ 241--243 and Example 14.4.8]{fulton}.

We first briefly recall how to work with differentials in algebraic geometry; see e.g.\ \cite[Chapter II.8]{Hartshorne}.
Given a $B$-module $A$, let $\Omega_{B/A}$ denote the $B$-module of relative K\"ahler differentials of $B$ over $A$. Given a scheme $V$ over $K$, let $\Omega_V := \Omega_{V/ \Spec K}$ denote the $\mathcal{O}_V$-module of relative differentials. When $V$ is a nonsingular variety, the \emph{cotangent bundle} $T^*V$ is the $K$-vector bundle $\Omega_V$.
Finally, given 
a map of $K$-schemes $f: V \to W$, let $\Omega_{V/W}$ denote the $\mathcal{O}_V$-module of relative differentials. Then $f$ induces a map of $\mathcal{O}_V$-modules $$\delta f : f^* \Omega_W \to \Omega_{V}.$$

\begin{defn} \label{defn_cf_fulton_style}
Let $f : X \to Y$ be a morphism of nonsingular varieties of the same dimension $r$. Let $T^*X$ and $T^*Y$ be the cotangent bundles. Then $f$ induces a homomorphism of vector bundles $\delta f : f^*(T^* Y) \to T^*X$ on $X$, and dually, a homomorphism of vector bundles $f': TX \to f^*(TY)$ on $X$. The map $f'$ is called the \emph{differential} of $f$. For each $x \in X$, we may identify the fiber $(f^*(TY))_x$ with $(TY)_{f(x)}$, thus viewing $f'_x$ as a linear map $(TX)_x \to (TY)_{f(x)}$. The \emph{critical locus} of $f$ is the set
\[C_f = \{x \in X(K) : f'_x : T_x(X) \to T_{f(x)}(Y) \text{ has rank} < r\}.\]
More generally, one may define the critical locus without the hypothesis on dimension by replacing $r$ by $\min \{\dim X, \dim Y\}$.

We now define the scheme structure on $C_f$ in the equidimensional case. By the assumptions on nonsingularity and dimension of $X$ and $Y$, the ranks of $TX$ and $TY$ are both $r$, and by definition the rank of $f^*(TY)$ and $TY$ are equal. So the domain and codomain of $f'$ are both rank $r$, allowing us to define the determinant $\det f'$. The determinant is a homomorphism of line bundles on $X$, and as such, on trivializing open sets $U$, it is given by multiplication by elements $a_U \in \mathcal{O}_X(U)$, each obtained as the determinant of any matrix in $\mathsf{M}_r(\mc{O}_X(U))$ representing $f'_U$.
The \emph{critical scheme} is the zero scheme of $\det f'$, or equivalently, the closed subscheme of $X$ determined by the sheaf of principal ideals generated by $a_U$. 
Since the determinant of a vector bundle homomorphism on $X$ vanishes precisely at the points $x \in X$ where the homomorphism has less than full rank, the underlying set of the critical scheme is indeed the critical locus $C_f$, and we henceforth use the notation $C_f$ to mean the scheme.
\end{defn}

Since $C_f$ is defined by a principal ideal sheaf, there are two mutually exclusive possibilities: either every irreducible component of $C_f$ has codimension $1$ in $X$ (i.e.\ $C_f$ is a divisor, a.k.a.\ hypersurface-with-multiplicity) or else $C_f = X$. 
These cases correspond to when $f$ is separable or inseparable, respectively. See Section \ref{sec:insep}.

There appears to be no standard definition of the critical scheme of a rational map. We use the following.

\begin{defn} \label{defn_cf_ratl}
Let $f : X \dashrightarrow Y$ be a rational map, with maximal domain of definition $U \subseteq X$. The \emph{critical scheme} of $f$ is defined as $C_f := C_{f|_U} \subseteq U$.
\end{defn}

The next Lemma describes the computation of the critical locus in the simplest possible setting. 
We use the following notation for the Jacobian matrix of a collection of rational functions: 
if $f_1, \ldots, f_r \in K(x_1, \ldots, x_r)$ 
then
\[
\frac{\partial(f_1, \ldots, f_r)}{\partial(x_1, \ldots, x_r)} :=
\begin{bmatrix}
\dfrac{\partial f_1}{\partial x_1} & \cdots & \dfrac{\partial f_1}{\partial x_r} \\ 
\vdots & \ddots & \vdots \\
\dfrac{\partial f_r}{\partial x_1} & \cdots & \dfrac{\partial f_r}{\partial x_r}
\end{bmatrix}.
\]

\begin{lem} \label{lem_computing_cf_for_opens}
Let $U$ be a nonempty open subset of $\AA^r = \Spec K[x_1, \ldots, x_r]$,
let $U'$ be a nonempty open subset of $\AA^r = \Spec K[y_1, \ldots, y_r]$, 
and let $f : U \to U'$ be a morphism given componentwise as $(f_1, \ldots, f_r)$. 
Then $C_f$ is the closed subscheme of $U$ defined by
\[
\det 
\frac{\partial(f_1, \ldots, f_r)}{\partial(x_1, \ldots, x_r)} = 0.
\]
\end{lem}
\begin{proof}
The Zariski cotangent space of the source and target are freely generated by $dx_1, \ldots, dx_r$ and $dy_1, \ldots, dy_r$ respectively. By definition we have 
$$(\delta f)(dy_i) = \sum_{j = 1}^r \frac{\partial f_i}{\partial x_j} dx_j.$$
In this choice of basis on source and target, the homomorphism $\delta f :  f^*(T^* U') \to T^* U$ is given by the matrix
\[ \tag{*}
\begin{bmatrix}
\dfrac{\partial f_1}{\partial x_1} & \ldots & \dfrac{\partial f_r}{\partial x_1} \\
\vdots & \ddots & \vdots \\
\dfrac{\partial f_1}{\partial x_r} & \ldots & \dfrac{\partial f_r}{\partial x_r} \\
\end{bmatrix}.
\]
Equipping $TU$ and $f^*(TU')$ with the dual bases, the matrix for $f'$ is the transpose of (*), and $C_f$ is the closed subscheme defined by its determinant.
\end{proof}

Lemma \ref{lem_computing_cf_for_opens} is sufficient for computing the critical scheme structure for rational self-maps of $\PP^n$, the setting of interest in this paper. Indeed, we can write any rational self-map of $\PP^n$ locally as a map $U \to U'$ of nonempty open subsets of $\AA^n$; then the scheme structure of $C_f$ is computed locally on each affine chart of the source $\PP^n \smallsetminus I_f$. The compatibility of the local expressions for $C_f$, and their independence of the choice of affine cover, are a consequence of the coordinate-free Definition \ref{defn_cf_fulton_style}.

\subsection{Homogeneous Jacobians}

Let us make the connection between the two definitions of the critical locus precise.

\begin{defn} \label{def_Jnd}
Let $n, d \ge 0$ be non-negative integers. 
For each multi-index $\alpha = (\alpha_0, \ldots, \alpha_n) \in \NN^{n+1}$ of degree $|\alpha| := \alpha_0 + \ldots + \alpha_n = d$ let $x^\alpha := x_0^{\alpha_0} \ldots x_n^{\alpha_n}$ be the corresponding monomial.
The \emph{universal tuple} of degree $d$ in dimension $n$ is the $(n+1)$-tuple defined by 
\begin{equation} \label{universal tuple}
    F_i(x_0, \ldots, x_n) := \sum_{|\alpha|=d} u_{i,\alpha} x^\alpha
    \in \ZZ[{\bf u}][{\bf x}]
    \qquad (0 \le i \le n)
\end{equation}
where the $u_{i,\alpha}$ are formal indeterminates.
The \emph{generic homogeneous Jacobian matrix} associated to the pair $(n,d)$ is given by
\[D_{n,d} := \frac{\partial(F_0, \ldots, F_n)}{\partial(x_0, \ldots, x_n)}
=
\begin{bmatrix}
    \dfrac{\partial F_0}{\partial x_0} & \ldots & \dfrac{\partial F_0}{\partial x_n} \\
    \vdots & \ddots & \vdots \\
    \dfrac{\partial F_n}{\partial x_0} & \ldots & \dfrac{\partial F_n}{\partial x_n}
\end{bmatrix}
\in 
\mathsf{M}_{n+1}(\ZZ[\mathbf{u}][\mathbf{x}]);
\]
its determinant 
\[
J_{n,d} := \det D_{n,d} \in \ZZ[{\bf u}][{\bf x}]
\]
is called the \emph{generic homogeneous Jacobian determinant}.
\end{defn}

\begin{eg}
Let $d = 1$. 
The universal linear tuple of projective $n$-space is 
\[
(u_{0,e_0} x_0 + \ldots + u_{0,e_n} x_n, 
\ldots, 
u_{n,e_0} x_0 + \ldots + u_{n,e_n} x_n)
\]
where the $e_j$'s are the standard basis multi-indices.
The generic homogeneous Jacobian matrix is 
\[
D_{n,1} = \begin{bmatrix}
    u_{0,e_0} & \ldots & u_{0,e_n} \\ 
    \vdots & \ddots & \vdots \\
    u_{n,e_0} & \ldots & u_{n,e_n}
\end{bmatrix},
\]
so the generic homogeneous Jacobian determinant is 
\[J_{n,1} = \det {(u_{i,e_j})}.\]
\end{eg} 

\begin{eg} \label{eg_quadratic}
Let $(n,d) = (1,2)$. The universal quadratic tuple of the projective line 
is
\[
(
u_{0,(2,0)} x_0^2 + u_{0,(1,1)} x_0 x_1 + u_{0,(0,2)} x_1^2, \,
u_{1,(2,0)} x_0^2 + u_{1,(1,1)} x_0 x_1 + u_{1,(0,2)} x_1^2
).
\]
The generic homogeneous Jacobian matrix is
\[
D_{1,2} = \begin{bmatrix}
    2 u_{0,(2,0)} x_0 + u_{0,(1,1)} x_1  &  u_{0,(1,1)} x_0 + 2 u_{0,(0,2)} x_1 \\
    2 u_{1,(2,0)} x_0 + u_{1,(1,1)} x_1  &  u_{1,(1,1)} x_0 + 2 u_{1,(0,2)} x_1
\end{bmatrix},
\]
so the generic homogeneous Jacobian determinant is
\begin{align*}
    J_{1,2} 
    &= 2 (u_{0,(2,0)} u_{1,(1,1)} - u_{0,(1,1)} u_{1,(2,0)}) x_0^2 \\
    &+ 4 (u_{0,(2,0)} u_{1,(0,2)} - u_{0,(0,2)} u_{1,(2,0)}) x_0 x_1 \\
    &+ 2 (u_{0,(1,1)} u_{1,(0,2)} - u_{0,(0,2)} u_{1,(1,1)}) x_1^2 .
\end{align*}
We observe that $2$ divides $J_{1,2}$.
\end{eg}

\begin{propn} \label{homogeneity properties}
$J_{n,d}$ is multi-homogeneous:
of degree 1 in each $i$-block of ${\bf u}$'s
(hence of degree $n+1$ in ${\bf u}$)
and, if $d \ge 1$, of degree $(n+1)(d-1)$ in ${\bf x}$.
\end{propn}

\begin{proof}
Scaling $F_i$ by $\lambda$ scales row $i$ of $D_{n,d}$ by $\lambda$, hence scales $J_{n,d}$ by $\lambda$ as well.
Next, since the entries are homogeneous of degree $d-1$ in $x$, scaling $x$ by $\lambda$ scales the whole matrix by $\lambda^{d-1}$, hence scales its determinant by $(\lambda^{d-1})^{n+1}$.
\end{proof}

\begin{defn}[homogeneous Jacobian of a tuple]
Let $R$ be a commutative unital ring and let $R[\mathbf{x}]_d$ denote the $R$-submodule of $R[\mathbf{x}]$ consisting of forms of degree $d$ (including 0).
For each $F = (F_0, \ldots, F_n) \in R[\mathbf{x}]_d^{n+1}$, we let $D_F$ and $J_F$ be obtained from $D_{n,d}$ and $J_{n,d}$ by interpreting $\ZZ$ in $R$ and evaluating the $u_{i,\alpha}$ at the parameters defining $F$. 
More precisely,
if $\theta_F : \ZZ[\mathbf{u}] \to R$ 
is the unique ring map defined by $u_{i,\alpha} \mapsto F_{i,\alpha}$ (the coefficient of $x^\alpha$ in $F_i$), and if we extend $\theta_F$ to the polynomial rings and thence to the matrix rings, then $D_F = \theta_F(D_{n,d})$ and $J_F = \theta_F(J_{n,d})$.
\end{defn}

Thus, each $F$ has an associated homogeneous polynomial $J_F \in R[\mathbf{x}]_{(n+1)(d-1)}$
and the assignment $F \rightsquigarrow J_F$ is itself homogeneous.

The critical loci of a map $f : \PP^n \to \PP^n$ and its lifts $F: \AA^{n+1} \to \AA^{n+1}$ are not \emph{a priori} related. While $J_F$ is precisely the polynomial defining $C_F$ in $\AA^{n+1}$ by Lemma \ref{lem_computing_cf_for_opens}, $C_f$ is a subscheme of $\PP^n$ defined only chartwise. One would hope that the projectivization $V_{\PP^n}(J_F)$ of $C_F$ would agree with $C_f$, but this is not always the case.

The following Proposition is well-known in characteristic $0$; a version of it (for birational maps of $\PP^2$) appears in \cite[\S 2.48]{Lamy}. We provide a general proof in arbitrary dimension.

\begin{propn} \label{Cf is VJF}
Let $f = [F] \in \End^n_d$. 
If $\Char K \nmid d$, then $V(J_F) = C_f$ as sets.
If $\Char K \mid d$, then $V(J_F) = \PP^n$.
\end{propn}

In fact, these equalities hold even on the level of schemes; this will follow from Theorem \ref{thm_flat_J_poly} below.

\begin{proof}
Regard $F$ as a map $\AA^{n+1} \to \AA^{n+1}$ and let $\pi : \AA^{n+1} \smallsetminus \{0\} \to \PP^n$ be the projection,
so $f \circ \pi = \pi \circ F$. 

First we examine the case $\Char K \nmid d$. This case will follow from the claim that, for all $p \in \AA^{n+1} \smallsetminus \{0\}$, we have $\rk F'_p = \rk f'_{\pi(p)} + 1$. 
Note that
\begin{align*}
    \rk f'_{\pi(p)} &= \rk {(f'_{\pi(p)} \circ \pi'_p)} & \text{(because $\pi'_p$ is surjective)}\\
    &= \rk {(\pi'_{F(p)} \circ F'_p)} & \text{(by the chain rule)} \\
    &= \rk F'_p - \dim {(\im F'_p \cap \ker \pi'_{F(p)})}.
\end{align*}
Consider the left-invariant vector field $w$ determined by the scaling action on $\AA^{n+1} \smallsetminus \{0\}$.
Since $F$ is homogeneous of degree $d$, it takes scaling-by-$\lambda$ to scaling-by-$\lambda^d$, so
\[F'_p (w_p) = \left.\frac{d}{d \lambda} F(\lambda p)  \right|_{\lambda=1} = \left.\frac{d}{d \lambda} \lambda^d F(p) \right|_{\lambda=1} = \left. d \cdot \lambda^{d-1} w_{F(p)} \right|_{\lambda=1} = d \cdot w_{F(p)} \tag{*}\]
which is nonzero by assumption on $K$. 
Since $\ker \pi'$ is generated by $w$,
we have $F'_p(w_p) \in \im F'_p \cap \ker \pi'_{F(p)}$
and so $\dim {(\im F'_{p} \cap \ker \pi'_{F(p)})} = 1$.

When $\Char K \mid d$, then (*) shows that the nonvanishing vector field $w$ is killed by $F'$, so $\det F'$ is identically $0$, implying that $V(J_F) = \PP^n$.
\end{proof}

The above proposition shows that the homogeneous Jacobian is defective in positive characteristic because a general map $f \in \End_d^n$ is separable and therefore satisfies $C_f \neq \PP^n$ (Section \ref{sec:insep}). 
As a result, in the wild case $C_f$ and $V(J_F)$ usually do not coincide, even as sets.

\section{Flat Jacobians} \label{sec:J*}
We have just seen that the homogeneous Jacobian determinant is an effective tool for computing critical loci of endomorphisms of $\PP^n$, but only in tame characteristics.
In this section, we introduce the key technical tool underlying the proofs of all our main results, the \emph{flat Jacobian} $J^*_{n,d}$. The flat Jacobian extends the useful properties of the homogeneous Jacobian to rational maps in arbitrary characteristic. The use of the word \emph{flat} is justified in Section \ref{sec:insep}, where we interpret $J^*_{n,d}$ as defining a flat family over a 
    certain 
parameter space of separable maps.

The main results of this section are collected in the following theorem. 
The reader should note that the $d = 0$ case is somewhat exceptional and that in our applications we usually take $d \geq 1$.
Also, we exclude the $n = 0$ case because the sole self-map of the one-point space $\PP^0$ does not have a well-defined algebraic degree.

\begin{thm} \label{thm_flat_J_poly}
Let $n \ge 1$ and $d \ge 0$ be integers.
There exists a unique polynomial
\[J^*_{n,d} \in \ZZ[{\bf u}][{\bf x}]\]
the \emph{generic flat Jacobian}, with the following properties:
\begin{enumerate}[label=(\alph*),ref=\thethm(\alph*)]
    \item \label{thm_flat_J_poly:a} $J_{n,d} = d J^*_{n,d}$.
    \item \label{thm_flat_J_poly:b} For any field $K$ and any rational map $f = [F] \in \Rat^n_d(K)$ defined over $K$, 
    \[C_f = V(J^*_F) \smallsetminus I_f\]
    as schemes, where $J^*_F := \theta_F(J^*_{n,d})$ is the specialization of $J^*_{n,d}$ at $F$ via the ring homomorphism
    \begin{align*}
        \theta_F : \ZZ[{\bf u}][{\bf x}] &\to K[\bf x] \\
        u_{i,\alpha} &\mapsto F_{i,\alpha}.
    \end{align*} 
\end{enumerate}
\end{thm}

For $f \in \End^n_d(K)$, 
Theorem \ref{thm_flat_J_poly:b} simplifies to Theorem \ref{thm_intro_Jstar} from the Introduction 
because $I_f = \varnothing$ for endomorphisms.

\begin{lem} \label{flat_J_unique}
For each pair $(n, d) \ne (0, 0)$
there exists at most one polynomial satisfying both (a) and (b) in Theorem \ref{thm_flat_J_poly}.
\end{lem}

\begin{proof}
If $d > 0$ and a polynomial satisfying (a) exists, then it is unique because $d$ is cancellable in $\ZZ$.
On the other hand, if $d = 0$ then condition (b) forces $J^*_{n,0} = 0$.
To see this, pick $K := \QQ(\mathbf{u})$ and let $F \in K[\mathbf{x}]_0^{n+1} = K^{n+1}$
be the universal tuple of degree $0$. Then $f := [F] \in \Rat^n_0(K)$ is a constant morphism,
so $C_f = \PP^n$ (because $n \ne 0$) and $I_f = \varnothing$.
By (b), $V(J^*_F) = \PP^n$, 
which means $J^*_F = 0$.
But $J^*_F = \theta_F(J^*_{n,0})$ and, by our choice of $K$, 
$\theta_F$ is just the inclusion map $\ZZ[{\bf u}][{\bf x}] \hookrightarrow \QQ({\bf u})[{\bf x}]$.
Thus $\theta_F(J^*_{n,0}) = J^*_{n,0}$ vanishes.
\end{proof} 

In light of Lemma \ref{flat_J_unique}, we make the following definition.

\begin{defn}[generic flat Jacobian]
For $n, d \ge 0$ define
\[
J^*_{n,d} 
:= 
\begin{dcases*} 0 & if $d = 0$ \\
    \frac{J_{n,d}}{d} & if $d \ge 1$ 
\end{dcases*}
\]
\end{defn}
as an element of $\QQ[{\bf u}][{\bf x}]$.
Note that $J^*_{n,d}$ inherits all the homogeneity properties of $J_{n,d}$ (cf.\ Proposition \ref{homogeneity properties}); 
in particular, $J^*_{n,d}$ is homogeneous of degree $n+1$ in $\bf u$ and of degree $(n+1)(d-1)$ in $\bf x$.

It is clear that with this definition,
$J^*_{n,0}$ satisfies
all the criteria of Theorem \ref{thm_flat_J_poly}.
The rest of this section is devoted to showing, when $d \ge 1$:
\begin{itemize}
\item that $J^*_{n,d}$ has \emph{integer} coefficients (Corollary \ref{d divides J});
\item that, when $F$ is in lowest terms, $\theta_F(J^*_{n,d})$ cuts out the critical scheme of the rational map defined by $F$ (Corollary \ref{Cf is VJF minus If}).
\end{itemize}

\subsection{Integrality}

For the remainder of Section \ref{sec:J*}, we work over an arbitrary (commutative, unital) ring $R$.

\begin{lem} \label{euler cramer}
Let $n, d \ge 0$ and let $F \in R[{\bf x}]^{n+1}_d$ be 
forms of the same degree.
For each $i = 0, \ldots, n$ put
\[
D^{(i)}_F
:= 
\begin{bmatrix}
    \partial_0 F_0 & \cdots & \partial_{i-1} F_0 & F_0 & 
    \partial_{i+1} F_0 & \cdots & \partial_n F_0 \\
    \vdots & \raisebox{0.44ex}{$\cdots$} & \vdots & \vdots & \vdots & \raisebox{0.44ex}{$\cdots$} & \vdots \\[.5ex]
    \partial_0 F_n & \cdots & \partial_{i-1} F_n & F_n & \partial_{i+1} F_n & \cdots & \partial_n F_n
\end{bmatrix},
\]
i.e.\ the Jacobian matrix of $F$ but with the $\partial_i$-column replaced by $F$ itself.
Then 
\[d \det D^{(i)}_F = x_i J_F.\]
\end{lem}

\begin{proof}
Let $D$ be the Jacobian matrix of $F$.
Viewing $F = (F_0, \ldots, F_n)$ and $x = (x_0, \ldots, x_n)$ as column vectors, 
Euler's theorem \ref{euler} implies
\[Dx = dF.\]
By Cramer's rule (valid over any commutative ring in the form $Av = w \implies (\det A)v_i = \det \Tilde{A}_{i|w}$ where $\Tilde{A}_{i|w}$ is the matrix $A$ with column $i$ replaced by $w$) 
we have 
\[
(\det D) x_i = \det \begin{bmatrix} \ \cdots & \partial_{i-1} F & dF & \partial_{i+1} F & \cdots \ \end{bmatrix}.
\]
But $\det D = J_F$ by definition, while the r.h.s.\ is $d \det D^{(i)}_F$ by linearity.
\end{proof}

\begin{cor}
Let $n, d \ge 0$.
\begin{enumerate}[label=(\alph*),ref=\thethm(\alph*)]
    \item \label{d divides J} 
    $d$ divides $J_{n,d}$ in $\ZZ[{\bf u}][{\bf x}]$.
    \item \label{Jstar formula} 
    If $(n, d) \ne (0, 0)$ 
    then the flat Jacobian of any $F \in R[{\bf x}]_d^{n+1}$ 
    may be calculated directly over $R$ 
    by any one of the expressions
    \[J^*_F = \frac{\det D^{(i)}_F}{x_i} \in R[{\bf x}] \qquad (i = 0, \ldots, n).\]
\end{enumerate}
\end{cor}

\begin{proof}
Applying Lemma \ref{euler cramer} to the universal tuple shows
\[d \det D^{(i)}_{n,d} = x_i J_{n,d} \tag{*}\]
for all $i$.
In particular, $d$ divides $x_0 J_{n,d}$ in $\ZZ[{\bf u}][{\bf x}]$.
If $d \ne 0$ then $d$ and $x_0$ are coprime, so
Euclid's lemma (valid in any GCD domain) implies $d$ divides $J_{n,d}$.
Together with the trivial case $J_{n,0} = 0$,
this proves part (a).
Next, we claim 
\[
\det D^{(i)}_{n,d} = x_i J^*_{n,d}. \tag{**}
\]
For $d \ne 0$ this is immediate from (*),
while if $d = 0$
then $n \ge 1$
and a direct calculation 
shows that 
\[
\det D^{(i)}_{n,0} 
= \det \begin{bmatrix} 
    \ \cdots & 0 & u_{0,{\bf 0}} & 0 & \cdots \ \\
    & \vdots & \vdots & \vdots \\
    \ \cdots & 0 & u_{n,{\bf 0}} & 0 & \cdots \ 
\end{bmatrix} 
= 0 
= x_i J^*_{n,0}.
\]
With (**) in hand, let $\theta_F : \ZZ[{\bf u}][{\bf x}] \to R[{\bf x}]$ be the specialization map. 
Then
\[
x_i J^*_F 
= x_i \theta_F(J^*_{n,d})
= \theta_F(x_i J^*_{n,d}) 
= \theta_F(\det D^{(i)}_{n,d})
= \det \theta_F(D^{(i)}_{n,d})
= \det D_F^{(i)}.
\]
Thus $x_i$ divides $\det D_F^{(i)}$ in $R[{\bf x}]$.
Since indeterminates in polynomial rings are never zero-divisors, the quotient is well-defined. Part (b) follows.
\end{proof}

\begin{rmk} \label{rmk as schemes}
Corollary \ref{d divides J} provides an alternate proof of 
    the second claim of 
Proposition \ref{Cf is VJF}. 
Corollary \ref{Jstar formula} does not hold if $(n, d) = (0,0)$ because $\det D^{(0)}_{0,0} = u_{0,{\bf 0}}$ is neither zero nor divisible by $x_0$.
Also, $\det D^{(i)}_F$ need not be divisible by $x_i$ if $F$ has unequal degree; 
for instance,
\[
\det \begin{bmatrix} x_0 + x_1 & 1 \\ x_0^2 + x_1^2 & 2x_1 \end{bmatrix} = 2x_0 x_1 + x_1^2 - x_0^2.
\]
\end{rmk}

\subsection{Criticality}

Next, we show that the flat Jacobian determines the critical scheme of any rational map. 
The key to the proof is the following determinant identity. 
We note that a related formula, with an alternate proof, appears in the 1882 exercise-book of Muir \cite[\S 192]{Muir}.

\begin{lem} \label{KEY LEMMA}
Let $F \in R[{\bf x}]^{n+1}$ be forms (not necessarily of the same degree). Then
\[F_j^{n+1} 
\det \frac{\partial(F_0/F_j, \ldots, F_{j-1}/F_j, F_{j+1}/F_j, \ldots, F_n/F_j)} {\partial(x_0, \ldots, x_{i-1}, x_{i+1}, \ldots, x_n)} = (-1)^{i+j} \det D^{(i)}_F 
\]
for all $0 \le i, j \le n$ such that $F_j$ is not a zero-divisor.
\end{lem}

\begin{proof}
We begin with the case $i = j = 0$.
Define the $n$-by-$(n+1)$ matrix
\[
\Pi := \begin{bmatrix} 
    -F_1   & F_0 \\
    \vdots &     & \ddots \\
    -F_n   &     &        & F_0
\end{bmatrix}
\]
and the $(n+1)$-by-$(n+1)$ matrix
\[
\Phi := D^{(0)}_F = \begin{bmatrix} 
    F_0    & \partial_1 F_0 & \cdots & \partial_n F_0 \\
    \vdots & \vdots         & \ddots &  \vdots        \\
    F_n    & \partial_1 F_n & \cdots & \partial_n F_n 
\end{bmatrix} 
\]
and multiply: 
\[
\Pi \Phi = \begin{bmatrix}
    0 & F_0 \partial_1 F_1 - F_1 \partial_1 F_0 & \cdots & F_0 \partial_n F_1 - F_1 \partial_n F_0 \\
    \vdots & \vdots & \ddots & \vdots \\
    0 & F_0 \partial_1 F_n - F_n \partial_1 F_0 & \ldots & F_0 \partial_n F_n - F_n \partial_n F_0
\end{bmatrix}.
\]
By the quotient rule, 
\[F_0 \partial_\ell F_k - F_k \partial_\ell F_0 = F_0^2 \cdot \partial_\ell (F_k/F_0)\]
for all $1 \le k, \ell \le n$. 
Thus the $(\mc{I},\mc{J})$-submatrix of $\Pi \Phi$, where $\mc{I} := \mc{J} := \{1, \ldots, n\}$,
has determinant equal to 
\[
\det {(\Pi\Phi)_{\mc{I},\mc{J}}} = 
F_0^{2n} \det \frac{\partial(F_1/F_0, \ldots, F_n/F_0)}{\partial(x_1, \ldots, x_n)}.
\]
On the other hand, the Cauchy--Binet formula says
\[
\det {(\Pi\Phi)_{\mc{I},\mc{J}}} = \sum_{\mc{K}} \det \Pi_{\mc{I},\mc{K}} \det \Phi_{\mc{K},\mc{J}}
\]
where $\mc{K}$ ranges over all size-$n$ subsets of $\{0, \ldots, n\}$.
Each such subset is the complement of a single number $k$.
It is easy to see that 
\[\det \Pi_{\mc{I}, \{k\}^c} = (-1)^k F_k F_0^{n-1}\]
for all $0 \le k \le n$: indeed, this is immediate for $k = 0$, and for $k \ge 1$ just expand along row $k$, whose sole nonzero entry is $-F_k$.
Meanwhile, because $\mc{J} = \{0\}^c$, 
\[\det \Phi_{\{k\}^c,\mc{J}} = \det \frac{\partial (F_0, \ldots, F_{k-1}, F_{k+1}, \ldots, F_n)}{\partial(x_1, \ldots, x_n)}.\]
It follows that 
\begin{align*}
    \det {(\Pi \Phi)_{\mc{I},\mc{J}}}
    &= F_0^{n-1} \sum_{k=0}^n (-1)^k F_k \det \frac{\partial (F_0, \ldots, F_{k-1}, F_{k+1}, \ldots, F_n)}{\partial(x_1, \ldots, x_n)} 
    \\
    &= F_0^{n-1} \det \Phi
\end{align*}
by recognizing the sum as a cofactor expansion, down the leftmost column of $\Phi$. 
The result (for $i = j = 0$) follows on cancelling $F_0^{n-1}$.

For $i > 0 = j$ we simply replace $D^{(0)}_F$ by $D^{(i)}_F$ in the definition of $\Phi$ and repeat the above calculation, using $\mc{J} := \{i\}^c$; this introduces a factor of $(-1)^i$.
Finally for $j > 0$ we observe that, by what we just proved, the l.h.s.\ is equal to 
\[(-1)^i \det D^{(i)}_{(F_j, F_0, \ldots, F_{j-1}, F_{j+1}, \ldots, F_n)}\]
where the subscript is a certain reordering $(F_{\tau(0)}, \ldots, F_{\tau(n)})$ of the original forms.
This merely permutes the rows of $\Phi$ through a cycle of length $j+1$,
so by properties of determinants we acquire a factor of $\sgn \tau = (-1)^j$. 
\end{proof}

The following Corollary completes the proof of Theorem \ref{thm_flat_J_poly}.

\begin{cor} \label{Cf is VJF minus If}
Let $n \ge 1$.
For all rational maps $f : \PP^n \dashrightarrow \PP^n$,
\[C_f = V(J^*_F) \smallsetminus I_f\]
as schemes, where $F$ is any lift of $f$. 
\end{cor}
\begin{proof}
It suffices to show that the sheaves of principal ideals defining $C_f$ and $V(J^*_F)$ agree on some open cover of $\PP^n \smallsetminus I_f$.
For each $i$ 
let $U_i \subset \PP^n$ be the affine space defined by inverting $x_i$, with coordinates 
$t_\ell := x_{\ell-1}/x_i$ for $\ell = 1, \ldots, i$
and $t_\ell := x_\ell/x_i$ for $\ell = i+1, \ldots, n$.
For each pair $i, j$ 
let $U_{ij} := U_i \cap f^{-1}(U_j)$.
Note that $\bigcup_j U_{ij} = U_i \smallsetminus I_f$ and that $\bigcup_{i,j} U_{ij} = \PP^n \smallsetminus I_f$.

Fix a lift $F$ of $f$.
In the open set $U_{ij}$ 
(where $x_i \ne 0$ and $F_j \ne 0$),
the rational map 
\[f = [F_0 : \ldots : F_n]\]
is given by the dehomogenization
\[f_{ij} : U_{ij} \to U_j\]
\[
f_{ij}(t_1, \ldots, t_n) 
:= \Big(\frac{F_0}{F_j}, \ldots, \frac{F_{j-1}}{F_j}, \frac{F_{j+1}}{F_j}, \ldots, \frac{F_n}{F_j}\Big)
(t_1, \ldots, t_i, 1, t_{i+1}, \ldots, t_n).\]
By Definition \ref{defn_cf_ratl}, $C_f \cap U_{ij}$ is cut out by $\det {(f_{ij})'}$.
By Lemma \ref{lem_computing_cf_for_opens}, 
the differential is represented by the $n$-by-$n$ matrix 
\begin{align*}
    (f_{ij})'\Big|_{(t_1, \ldots, t_n)}
    &:=
    \frac{\partial f_{ij}}{\partial(t_1, \ldots, t_n)}
    (t_1, \ldots, t_n)
    \\
    &\phantom{:}=
    \frac{\partial (F_k/F_j : k \ne j)}{\partial(x_\ell : \ell \ne i)}(t_1, \ldots, t_i, 1, t_{i+1}, \ldots, t_n).
\end{align*}
Taking determinants and using Lemma \ref{KEY LEMMA} plus Corollary \ref{Jstar formula}, we get
\begin{align*}
    \det {(f_{ij})'}
    \Big|_{(t_1, \ldots, t_n)}
    &= (-1)^{i+j} \cdot 1 \cdot \frac{J^*_F}{F_j^{n+1}}(t_1, \ldots, t_i, 1, t_{i+1}, \ldots, t_n). 
\end{align*}
Substituting
\[(t_1, \ldots, t_n) = \Big(\frac{x_0}{x_i}, \ldots, \frac{x_{i-1}}{x_i}, \frac{x_{i+1}}{x_i}, \ldots, \frac{x_n}{x_i}\Big)\]
we observe that 
\begin{align*}
    \frac{J^*_F}{F_j^{n+1}}
    \Big(\frac{x_0}{x_i}, \ldots, \frac{x_{i-1}}{x_i}, 1, \frac{x_{i+1}}{x_i}, \ldots, \frac{x_n}{x_i}\Big)
    &= x_i^{(n+1)d - (n+1)(d-1)} \frac{J^*_F}{F_j^{n+1}}(x_0, \ldots, x_n) \\
    &= \Big( \frac{x_i}{F_j(x)} \Big)^{n+1} J^*_F(x)
\end{align*}
because $J^*_F$ and $F_j$ are homogeneous of degrees $(n+1)(d-1)$  and $d$, respectively.
It follows that
\[
\det {(f_{ij})'}\Big|_{\big(\ldots, \frac{x_{i-1}}{x_i}, \frac{x_{i+1}}{x_i}, \ldots\big)}
=
(-1)^{i+j} 
\Big( \frac{x_i}{F_j(x)}\Big)^{n+1} J^*_F(x)
\]
as an identity in $K[{\bf x}][\frac{1}{x_i}, \frac{1}{F_j}]$.
This shows that $\det {(f_{ij})'}$ and $J^*_F$ differ only by a unit multiple, so they define the same principal ideal.
Thus the schemes $C_f$ and $V(J^*_F)$ coincide on $U_{ij}$.
\end{proof}

\begin{rmk}
The special case of Corollary \ref{Cf is VJF minus If} for birational maps of the plane ($n=2$)
appears implicitly in \cite[Theorem 3.5.6]{AlberichCarraminana} 
and explicitly in \cite[\S 2.48]{Lamy}.
The latter obtains (in our notation) the identity 
\[\det {(f_{22})}'\Big|_{(x_0, x_1)} = \frac{J_F}{d \cdot F_2^3}(x_0, x_1, 1)\]
via Euler's theorem (as in Lemma \ref{euler cramer}) 
and a quotient-rule calculation (as in Lemma \ref{KEY LEMMA}),
using this to show that $V(J_F) = \operatorname{Exc}(f)$ for $f \in \operatorname{Bir}^2_d$ when $\Char K \nmid d$.
\end{rmk}

\subsection{Further properties of $J^*$}
We conclude this section by collecting a few other valuable properties of the flat Jacobian.
Since $J = dJ^*$, each of these statements implies a corresponding statement involving the homogeneous Jacobian.

\begin{propn}
Let $\varphi : R \to S$ be a homomorphism of rings;
let $F, G \in R[{\bf x}]^{n+1}$ be tuples; 
let $\Phi \in R[{\bf x}]$ be a form.
\begin{enumerate}[label=(\alph*),ref=\thethm(\alph*)]
    \item \label{naturality} \emph{naturality:} $J^*_{\varphi(F)} = \varphi(J^*_F)$.
    \item \label{strong homogeneity} \emph{strong homogeneity:} $J^*_{\Phi F} = \Phi^{n+1} J^*_F$.
    \item \label{chain rule} \emph{chain rule:}
    $J^*_{F \circ G} = (J^*_F \circ G) J^*_G$.
\end{enumerate}
\end{propn}
Note that part (b) generalizes \cite[Lemma 2.3]{BlancHeden} to arbitrary characteristic,
and part (c) is used in \cite[proof of Proposition 2.7(1)]{BlancHeden}.
\begin{proof}
\hfill 
\begin{enumerate}[label=(\alph*)]
    \item Let $\theta_F : \ZZ[{\bf u}] \to R$ 
    and $\theta_{\varphi(F)} : \ZZ[{\bf u}] \to S$ 
    be the specialization maps 
    sending $u_{i,\alpha}$ to $F_{i,\alpha}$ and $\varphi(F_{i,\alpha})$, respectively.
    By uniqueness, $\theta_{\varphi(F)} = \varphi \circ \theta_F$ (the triangle commutes).
    Now evaluate at $J^*_{n,d}$.
    \item By Corollary \ref{Jstar formula}, the product rule, and multilinearity of the determinant,
    \begin{align*}
        x_0 J^*_{\Phi F} &= 
        \det \begin{bmatrix} \Phi F & \partial_1 (\Phi F) & \cdots & \partial_n (\Phi F) \end{bmatrix}
        \\ &= \det \begin{bmatrix} \Phi F & \Phi \partial_1 F + F \partial_1 \Phi & \cdots & \Phi \partial_n F + F \partial_n \Phi \end{bmatrix}
        \\ &= \Phi \det \begin{bmatrix} F & \Phi \partial_1 F + F \partial_1 \Phi & \cdots & \Phi \partial_n F + F \partial_n \Phi \end{bmatrix}
        \\ &= \Phi \det \begin{bmatrix} F & \Phi \partial_1 F & \cdots & \Phi \partial_n F \end{bmatrix} 
        \\ &= \Phi \Phi^n \det \begin{bmatrix} F & \partial_1 F & \cdots & \partial_n F \end{bmatrix} 
        \\ &= \Phi^{n+1} x_0 J^*_F.
    \end{align*}
    \item Exercise. (Calculate $J_{F \circ G}$ for universal $F, G \in \ZZ[{\bf u}, {\bf v}][{\bf x}]^{n+1}$ first.)
    \qedhere
\end{enumerate}
\end{proof}

\section{The Examples} \label{sec:eg}

We assume 
\[\boxed{n, d \ge 2}\]
throughout this section.

\subsection{Specific implies general}

Let
\[
\mathcal{E}_{n,d} := \{f \in \End^n_d : C_f \text{ is an integral hypersurface}\}.
\]
We begin by showing that
in order to prove our main Theorem \ref{main} (that $\mathcal{E}_{n,d}$ is dense open), 
it suffices to exhibit (for each $n$ and $d$) just one example of a tuple $F$ with absolutely irreducible flat Jacobian.

\begin{lem} \label{lem_specific_implies_general}
Let $K$ be an algebraically closed field.
Suppose that the set 
\[O := \{[F] \in \Bar\End^n_d : J_F^* \text{ is irreducible}\}\]
is nonempty.
Then $\mathcal{E}_{n,d}$ is a Zariski-dense open subset of $\End^n_d$. 
\end{lem}

\begin{proof}
By Theorem \ref{thm_flat_J_poly:b}, 
\[\mathcal{E}_{n,d} = {\End^n_d} \cap O.\]
Since $\Bar\End^n_d$ is irreducible, 
``dense open'' is equivalent to ``nonempty open'', 
and the class of dense opens is closed under intersection. 
Thus, $\mathcal{E}_{n,d}$ is dense open if and only if $O$ is nonempty open. 

By hypothesis, $O$ is nonempty; 
to show $O$ is open, we argue as follows.
Write $J^*_{n,d} \in \ZZ[\bf u][\bf x]$ as 
\[J^*_{n,d} = \sum_{|\beta|=m} \Xi_\beta {\bf x}^\beta\]
where $m = (n+1)(d-1)$ and the $\Xi_\beta$ are homogeneous forms of degree $n+1$ in the variables $\mathbf{u} = u_{i,\alpha}$ (the coordinates on $\Bar\End^n_d$).
Let $U := \Bar\End^n_d \smallsetminus V(\Xi_{\beta} : |\beta| = m)$
and let $j_0 : U \to \F^n_m$ be the function sending $[F]$ to $[J^*_F]$. 
It is clear that $j_0$ is well-defined: 
on $U$, $J^*_F$ is a nonzero form of degree $m$ in $n+1$ variables, 
so $[J^*_F] \in \F^n_m$; 
and for any nonzero scalar $\lambda$, 
$J^*_{\lambda F} = \lambda^{n+1} J^*_F$, so $[J^*_{\lambda F}] = [J^*_F]$ (cf.\ Proposition \ref{homogeneity properties}).

We claim that 
\[O = j_0^{-1}(\H^n_m).\]
By definition, $[F] \in O$ if and only if $[J^*_F] \in \H^n_m$; 
and since irreducibles are \emph{a fortiori} nonzero, $O \subseteq U$. 
This lets us write $[J^*_F]$ as $j_0([F])$, 
so that $[J^*_F] \in \H^n_m$ if and only if $[F] \in j_0^{-1}(\H^n_m)$---as desired. 
Finally, since $j_0$ is continuous and $\H^n_m$ is open (by Proposition \ref{Hnd nonempty}), $O$ is open.
\end{proof}

\subsection{The search for examples} \label{sec:eg motivation}

Before presenting the examples, 
we discuss some issues that arise in their construction.
Many of these difficulties already figure in the tame case $\Char K \nmid d$,
where $J^*_F$ and $J_F$ are associates and hence factor identically,
so with little loss of generality 
we phrase our informal discussion in terms of $J_F$.

Finding an explicit tuple $F$ such that $J_F$ is irreducible is nontrivial.
While one might hope to start with an obviously-irreducible determinant and back-solve for a tuple $F$ realizing it, this is prohibitively difficult because not every matrix of forms is a Jacobian.
Indeed, a na\"ive dimension-count shows that the $K$-linear function
\begin{align*}
K[{\bf x}]_d^{n+1} &\to \mathsf{M}_{n+1}(K[{\bf x}]_{d-1})
\\
F &\mapsto D_F
\end{align*}
is far from surjective.
Thus, we are forced to try various $F$'s and see what works.

Thanks to Lemma \ref{lem_specific_implies_general}, we need not worry about common zeroes. 
That said, Proposition \ref{strong homogeneity} implies that $J_F$ is divisible by $\gcd(F)$, so at the very least $F$ must be a rational map.
However, many well-studied classes of maps are precluded by the following caveats.
\begin{enumerate}
\item $F$ cannot be a monomial map, because
\begin{equation*} \label{eq_monomial}
    J_{x^A} = (\det A) x^{A_0 + \ldots + A_n - (1, \ldots, 1)}
\end{equation*}
where $A_0, \ldots, A_n$ are the rows of $A \in \mathsf{M}_{n+1}(\ZZ)$. 
\item $F$ cannot be ``minimally critical'' in the sense of \cite{Ingram} because 
\[J_{Ax^d} = (\det A) d^{n+1} (x_0 \ldots x_n)^{d-1}\]
for $A \in \mathsf{M}_{n+1}(K)$.
\item $F$ cannot be a symmetric power \cite[Proposition 7.1]{Ole}.
\item It can be shown that if any $F_i$ is divisible by the square of a form $\Phi$, 
then $J_F$ is divisible by $\Phi$.
Thus, every component of $F$ must be squarefree.
In particular, $F$ cannot be a polynomial map (i.e.\ the extension of a self-map of some $\AA^n \subset \PP^n$);
and if $d > n+1$ then $F$ cannot have any monomial component at all.
\end{enumerate}
While the simplicity of their form makes such maps attractive candidates to study in other contexts (e.g.\ dynamics),
it is precisely that simplicity which makes them unsuitable for our goal.
The moral is, $F$ cannot be too simple lest $J_F$ be reducible.


On the other hand, $F$ cannot be too complicated. 
Yes, any sufficiently complex example ought to work, 
but for general $F$ every coefficient of $J_F$ is nonzero, 
and \emph{proving} the absolute irreducibility of a ``random'' high-degree polynomial in many variables
is not easy.
The types of polynomials that arise as $J_F$'s do not seem amenable to standard irreducibility tricks like Eisenstein's criterion or Newton polygons, 
which require fine control over the valuations of the coefficients 
of the polynomial at hand.
Nor did we wish to check Ruppert's equations \cite[\S 3.2]{Schinzel} for the reducible locus in $\F^n_m$, as those are far too numerous, are only given explicitly in characteristic 0, and are better suited to fixed $n$ and $m$ anyway.
Lastly, we desired a elementary proof that could be verified manually. 

These considerations led us to settle on the following \textit{ad hoc} strategy for establishing irreducibility:
show that the specialization $J_F(x_0, x_1, x_2, 0, \ldots, 0)$ is linear in $x_2$, and apply Lemmas \ref{geo irr} and \ref{lin irr}.
Geometrically, this is analogous to intersecting $C_f$ with the linear subvariety $V(x_3, \ldots, x_n) \cong \PP^2$ and showing that the resulting plane curve is birational to $\PP^1$ via the projection $(x_0 : x_1 : x_2) \mapsto (x_0 : x_1)$.

Still, this proof strategy relies on $J_F$ having a readily computable ``closed form''.
Naturality of the determinant means that $J_F(x_0, x_1, x_2, 0, \ldots, 0)$ can be computed \emph{without} full knowledge of $J_F$, 
and setting so many variables to 0 causes many entries of $D_F$ to vanish. 
While this sparsity does already facilitate the calculation of the determinant somewhat, 
it does not ensure linearity in $x_2$.

The key to the latter lies in picking $F$ so that $D_F$ is primarily made up of repeating elements: pure powers of $x_0$ and $x_1$, arranged in a ``quasi-circulant'' structure, along just two diagonals, with the remaining variables confined to essentially just two columns.
This is achieved using the trivial observations that
\begin{itemize}
\item if $F_i$ does not depend on $x_j$, then $\partial_j F_i$ is zero; and
\item if $F_i$ is linear in $x_j$, then $\partial_j F_i$ is the coefficient of $x_j$ in $F_i$.
\end{itemize}
Ultimately, we land on a tuple $F$ such that $F_1, \ldots, F_n$ are binomials and every $F_i$ is concentrated on $x_0, x_1$ (i.e.\ has degree $\le 1$ in the other variables), leaving $F_0$ otherwise generic.

Our carefully chosen $F$ lets us obtain
\[J_F(x_0, x_1, x_2, 0, \ldots, 0) = P(x_0, x_1) + Q(x_0, x_1) x_2\]
for some explicit polynomials $P$ and $Q$ depending on $F_0$.
Now the application of Lemma \ref{lin irr} requires the verification 
that $\gcd(P, Q) = 1$, 
i.e.\ 
that $P$ and $Q$ have no common projective roots. 
Fortunately, $P$ factors nicely, so this verification is relatively trivial;
but for technical reasons we expedite the task using the strange resultant--discriminant identity of Lemma \ref{res jac disc}.
Keeping $F_0$ generic affords us enough flexibility to ensure $\Res(P, Q) \ne 0$, 
which proves $J_F$ is irreducible.

\subsection{The examples}

\begin{propn} \label{unified eg}
Let $e \in \{0, 1\}$
and define $F \in K[{\bf x}]_d^{n+1}$ by 
\begin{align*}
    F_0 &:= G(x_0, x_1) x_2^e \\
    F_i &:= x_1^{d-1} x_i - x_0^{d-1} x_{i+1} & (1 \le i \le n)
\end{align*}
where $x_{n+1} = x_0$ and $G$ is homogeneous of degree $d-e$.
Assume $\Char K \nmid d - 1 + e$.
If 
\begin{enumerate}
    \item the coefficients of $x_0 x_1^{d-e-1}$ and $x_0^{d-e-1} x_1$ in $G$ are nonzero,
    \item $G$ has no repeated projective roots, and
    \item $G$ has no projective roots in common with $x_1^{n(d-1)+1} - x_0^{n(d-1)+1}$,
\end{enumerate}
then $J^*_F$ is (absolutely) irreducible.
\end{propn}

The proof strategy sketched in \S\ref{sec:eg motivation} goes through without a hitch in characteristic 0, but may fail if $\Char K$ divides $d$ or $d-1$; yet the positive-characteristic case is conceptually identical.
Compared to the sketch, the formal proof given below is a bit more convoluted in order to handle all characteristics simultaneously.
Moreover, to minimize dependence on new material, we utilize the familiar Jacobian $J_F$ as long as possible; this requires working over $\ZZ$ first.
Readers comfortable with the flat Jacobian may calculate $J^*_F$ directly over $K$ using Corollary \ref{Jstar formula}.

\begin{proof}[Proof of Proposition \ref{unified eg}]
For our initial calculations we work over $\ZZ$, viewing the coefficients of $G$ as indeterminates ${\bf v} = v_0, \ldots, v_{d-e}$.
Let 
\[G := \sum_{i=0}^{d-e} v_i x_0^i x_1^{d-e-i}\]
and define $F_0, \ldots, F_n \in \ZZ[{\bf v}][{\bf x}]$ 
using the expressions in the statement of the Proposition.
Observe that for each $i > 0$, the polynomial $F_i$ depends only on the variables $x_0, x_1, x_i, x_{i+1}$. 
Taking partial derivatives, the Jacobian matrix $D_F$ of $F$ is
\small
\[
\begin{bmatrix}
    A x_2^e & B x_2^e & eGx_2^{e-1} & 0 & 0 & \cdots & 0 & 0 \\
    (1-d)x_0^{d-2} x_2 & d x_1^{d-1} & -x_0^{d-1} & 0 & 0 & \cdots & 0 & 0 \\
    (1-d)x_0^{d-2} x_3 & (d-1)x_1^{d-2} x_2 & x_1^{d-1} & -x_0^{d-1} & 0 & \cdots & 0 & 0 \\
    (1-d)x_0^{d-2} x_4 & (d-1)x_1^{d-2} x_3 & 0 & x_1^{d-1} & -x_0^{d-1} & \cdots & 0 & 0 \\
    (1-d)x_0^{d-2} x_5 & (d-1)x_1^{d-2} x_4 & 0 & 0 & x_1^{d-1} & \cdots & 0 & 0 \\
    \vdots & \vdots & \vdots & \vdots & \vdots & \ddots & \vdots & \vdots \\
    (1-d)x_0^{d-2} x_n & (d-1)x_1^{d-2}x_{n-1} & 0 & 0 & 0 & \cdots & x_1^{d-1} & -x_0^{d-1} \\
    -dx_0^{d-1} & (d-1)x_1^{d-2}x_n & 0 & 0 & 0 & \cdots & 0 & x_1^{d-1}
\end{bmatrix}\]
\normalsize
where $A := \partial_0 G$ and $B := \partial_1 G$ for ease of notation.
Let 
\[\Tilde{J_F}(x_0, x_1, x_2) := J_F(x_0, x_1, x_2, 0, \ldots, 0).\]
By naturality of the determinant, 
\small 
\[
\Tilde{J_F}  =  \begin{vmatrix}
    A x_2^e & B x_2^e & eGx_2^{e-1} & 0 & 0 & \cdots & 0 & 0 \\
    (1-d)x_0^{d-2} x_2 & d x_1^{d-1} & -x_0^{d-1} & 0 & 0 & \cdots & 0 & 0 \\
    0 & (d-1)x_1^{d-2} x_2 & x_1^{d-1} & -x_0^{d-1} & 0 & \cdots & 0 & 0 \\
    0 & 0 & 0 & x_1^{d-1} & -x_0^{d-1} & \cdots & 0 & 0 \\
    0 & 0 & 0 & 0 & x_1^{d-1} & \cdots & 0 & 0 \\
    \vdots & \vdots & \vdots & \vdots & \vdots & \ddots & \vdots & \vdots \\
    0 & 0 & 0 & 0 & 0 & \cdots & x_1^{d-1} & -x_0^{d-1} \\
    -dx_0^{d-1} & 0 & 0 & 0 & 0 & \cdots & 0 & x_1^{d-1}
\end{vmatrix}.
\]
\normalsize 
Expanding down the leftmost column and noting that the cofactors are block triangular,
we obtain 
\begin{align*}
    \Tilde{J_F} &= 
    Ax_2^e \big(dx_1^{2(d-1)} + (d-1)x_0^{d-1} x_1^{d-2} x_2\big) (x_1^{d-1})^{n-2} 
    \\
    & - (1-d) x_0^{d-2} x_2 \big(B x_1^{d-1} x_2^e - (d-1)e G x_1^{d-2} x_2^e \big) (x_1^{d-1})^{n-2} 
    \\
    & + (-1)^n (-dx_0^{d-1}) \big({-B x_0^{d-1}} x_2^e - de G x_1^{d-1} x_2^{e-1} \big)(-x_0^{d-1})^{n-2}.
\end{align*}
Note that all signs cancel in the last line, giving
\[
\ldots + dx_0^{d-1} \big(B x_0^{d-1} x_2^e + de G x_1^{d-1} x_2^{e-1} \big)(x_0^{d-1})^{n-2}.
\]
Next, consider $\Tilde{J_F}$ as a polynomial in $x_2$ with coefficients in $\ZZ[{\bf v}][x_0, x_1]$:
\begin{align*}
    \Tilde{J_F} 
    &= (d-1) x_0^{d-2} x_1^{(n-1)(d-1)-1} \big(
    A x_0 + B x_1 - (d-1) e G
    \big) x_2^{e+1} \\
    &+ d\big( 
    A x_1^{n(d-1)} + B x_0^{n(d-1)}
    \big) x_2^e 
    \, + \, 
    \big(d^2e G x_0^{(n-1)(d-1)} x_1^{d-1}\big) x_2^{e-1}.
\end{align*}
Since $G$ is homogeneous of degree $d-e$ in $x_0$ and $x_1$, Euler's theorem \ref{euler} implies 
\[
A x_0 + B x_1 - (d-1) e G
= (d-e)G - (d-1)eG
= d(1-e) G.
\]
Thus each term of $\Tilde{J_F}$ is divisible by $d$ as expected, 
and so
\begin{align*}
    \Tilde{J^*_F}(x_0, x_1, x_2) 
    &:= J^*_F(x_0, x_1, x_2, 0, \ldots, 0) \\
    &\phantom{:}= \big( (d-1)(1-e) G x_0^{d-2} x_1^{(n-1)(d-1)-1}\big) x_2^{e+1} \\
    &\phantom{:}+ \big( A x_1^{n(d-1)} + B x_0^{n(d-1)} \big) x_2^e 
    \, + \, 
    \big(de G x_0^{(n-1)(d-1)} x_1^{d-1}\big) x_2^{e-1}
\end{align*}
by Theorem \ref{thm_flat_J_poly:a} 
plus commutativity of evaluation maps affecting disjoint sets of variables.
Although \emph{prima facie} $\Tilde{J^*_F}$ has 3 terms, it is in fact linear in $x_2$.
Indeed, if we write
\begin{align*}
P_0 &:= (d-1)Gx_0^{d-2} x_1^{(n-1)(d-1)-1} 
\\
P_1 &:= dGx_0^{(n-1)(d-1)}x_1^{d-1}
\shortintertext{and}
Q &:= A x_1^{n(d-1)} + B x_0^{n(d-1)}
\end{align*}
then
\[
\Tilde{J^*_F} 
= (1-e)P_0 x_2^{e+1} + Q x_2^e + eP_1 x_2^{e-1}
= 
\begin{dcases*}
    P_0 x_2 
    + Q & if $e = 0$, \\
    Q x_2 
    + P_1 & if $e = 1$.
\end{dcases*}
\]
Next, we calculate the resultant 
of $P_e$ and $Q$ in $\ZZ[{\bf v}]$.
Since $P_e$ is a product of $d-1+e$ and $G$ with powers of $x_0$ and $x_1$, 
homogeneity and multiplicativity entail
\[
\Res(P_e, Q) 
= 
(d-1+e)^{N_1} \Res(x_0, Q)^{N_2} \Res(x_1, Q)^{N_3} \Res(G, Q) 
\]
for some exponents $N_i \ge 0$.
Easily, 
\[\Res(x_0, Q) = Q(0, 1) = \partial_0 G(0, 1) \quad\text{and}\quad \Res(x_1, Q) = Q(-1, 0) = \pm \partial_1 G(1, 0)\]
while $\Res(G, Q)$ 
is given by Lemma \ref{res jac disc}.
It follows that $\Res(P_e, Q)$ is equal to
\[\pm (d-1+e)^{N_1} \partial_0 G(0, 1)^{N_2} \partial_1 G(1, 0)^{N_3} \Disc(G) \Res(G, x_1^{n(d-1)+1} - x_0^{n(d-1)+1})\]
in $\ZZ[{\bf v}]$.
This completes our explicit calculations.

To conclude the proof,
let $\rho : \ZZ[{\bf v}] \to K$ be the ring homomorphism sending $v_i$ to the corresponding coefficient of $G$ in $K$.
By our hypotheses (1--3) on $G$ and our assumption on $\Char K$,
\[\Res(\rho(P_e), \rho(Q)) \ne 0.\]
Moreover, $m-1 = (n+1)(d-1)-1 \ge 2$ since $n, d \ge 2$.
Thus $\rho(P_e)$ and $\rho(Q)$ are nonzero, and they have no common projective roots over $K = \Bar K$, i.e.\ they are coprime in $K[x_0, x_1]$.
By Lemma \ref{lin irr},
$\rho(\Tilde{J^*_F})$, which equals one of
\[\rho(P_0)x_2 + \rho(Q) \quad\text{or}\quad \rho(Q)x_2 + \rho(P_1),\]
is irreducible in $K[x_0, x_1, x_2]$.
But by naturality (\ref{naturality}), $\rho(\Tilde{J^*_F}) = \Tilde{J^*_{\rho(F)}}$.
Thus by Lemma \ref{geo irr}, $J^*_{\rho(F)}$ is irreducible in $K[{\bf x}]$.
\end{proof}

Combining the reduction step with the key example, we prove our main result.

\begin{proof}[Proof of Theorem \ref{main}]
Proposition \ref{unified eg} exhibits, in each dimension $n \ge 2$ and degree $d \ge 2$, 
two families of rational maps whose critical scheme is integral
in all but finitely many characteristics.
Since $\Char K$ cannot divide both $d$ and $d-1$, 
we see that the set
\[O = \{[F] \in \Bar\End^n_d : J^*_F \text{ is irreducible}\}\]
from Lemma \ref{lem_specific_implies_general}
is non-empty in all cases.
To be totally explicit: 
in the tame case $\Char K \nmid d$, take $e = 1$;
while in the wild case $\Char K \mid d$ we have $\Char K \nmid d-1$, so take $e = 0$.
\end{proof}

\section{From rational maps to morphisms} \label{sec:morphisms}
Our examples are far from morphisms, because $I_f \supseteq V(x_0, x_1)$ and so $\codim I_f = 2$.
In this section we ``interpolate'' between our example and the power map to produce a \emph{morphism} with irreducible critical locus.

\begin{propn} \label{lin comb}
Let $[F] \in \Bar\End^n_d$ be such that $J^*_F$ is irreducible, and let $[\Pi] \in \End^n_d$ be a morphism.
Then for all but finitely many $[s:t] \in \PP^1$, 
the linear combination 
\[sF + t\Pi = (sF_0 + t\Pi_0, \ldots, sF_n + t\Pi_n)\]
defines a degree-$d$ endomorphism of $\PP^n$ whose critical locus is an integral hypersurface.
\end{propn}
\begin{proof}
Let $L \subset \Bar\End^n_d$ be the projective line through the points $[F]$ and $[\Pi]$ (w.l.o.g.\ distinct).
Let $O$ be the set from Lemma \ref{lem_specific_implies_general}. 
By the proof of said Lemma, $O$ is open. 
Thus $L \cap O$ and $L \cap {\End^n_d}$ are relatively open subsets of $L$.
Moreover, $[F] \in L \cap O$ and $[\Pi] \in L \cap {\End^n_d}$ by construction, so these sets are non-empty.
Since $L$ is irreducible, the intersection $(L \cap O) \cap (L \cap {\End^n_d}) = L \cap \mathcal{E}_{n,d}$ is again relatively open and non-empty.
It follows that the set
\[\{[s:t] \in \PP^1 : [sF + t\Pi] \in \mathcal{E}_{n,d}\},\]
that is, the preimage of $L \cap \mathcal{E}_{n,d}$ under the natural parametrization $\PP^1 \overset{\sim}{\to} L$,
is a non-empty Zariski-open subset of $\PP^1$.
In particular, it is cofinite.
\end{proof}

It can be shown that the exceptional $[s:t]$'s are algebraic over any common field of definition of $F$ and $\Pi$. 
Together, Propositions \ref{lin comb} and \ref{unified eg} immediately yield morphism-examples over $\CC$.
We now exhibit a morphism-example over $\QQ$.

\begin{thm}[= Theorem \ref{morphism example intro}] \label{morphism example}
Let $n, d \ge 2$ and let $p$ be an odd prime not dividing $d-1$.
Define
\begin{align*}
    F_0 &:= px_0^d + 2x_0 x_1 \frac{x_1^{d-1}-x_0^{d-1}}{x_1-x_0} \\
    F_i &:= px_i^d + 2x_1^{d-1} x_i - 2x_0^{d-1} x_{i+1} \qquad (1 \le i \le n)
\end{align*}
where $x_{n+1} = x_0$.
Then 
\[f := [F_0 : \ldots : F_n]\]
is a degree-$d$ endomorphism of $\PP^n$ defined over $\QQ$ whose critical scheme $C_f$ is an integral hypersurface.
\end{thm}
\begin{proof}
Consider the reduction maps
\[I_f(\Bar\QQ) \to I_f(\Bar\FF_2) \quad\text{and}\quad C_f(\Bar\QQ) \to C_f(\Bar\FF_p).\]
Because $p$ is odd,
\[F \equiv (x_0^d, x_1^d, \ldots, x_n^d) \pmod 2\]
which implies $I_f(\Bar\FF_2) = \varnothing$;
thus $I_f(\Bar\QQ) = \varnothing$, so $f$ is a morphism.
Meanwhile, modulo $p$ and up to unit multiples, 
$F$ is congruent to a tuple of the shape considered in Proposition \ref{unified eg}:
namely, with $e = 0$ and 
\[G(x_0, x_1) := x_0 x_1 \frac{x_1^{d-1} - x_0^{d-1}}{x_1 - x_0}
= x_0^{d-1} x_1 + x_0^{d-2} x_1^2 + \ldots + x_0 x_1^{d-1}.\]
Over $\Bar\FF_p$, the polynomial $G$ 
    has the following properties:
\begin{enumerate}
    \item The coefficients of $x_0 x_1^{d-1}$ and $x_0^{d-1} x_1$ are each $1$.
    \item The projective roots of $G$ are $[1:0]$, $[0:1]$, and $[\zeta : 1]$ for each $\zeta \neq 1$ satisfying $\zeta^{d-1} = 1$. Since $p \nmid d-1$, the roots of $z^{d-1}-1$ in $\Bar\FF_p$ are pairwise distinct. Thus $G$ has no repeated projective roots.
    \item Direct evaluation shows that $[1:0]$ and $[0:1]$ are not projective roots of $x_1^{n(d-1)+1} - x_0^{n(d-1)+1}$.
    Neither are the remaining projective roots of $G$, since if $\zeta \neq 1$ satisfies $\zeta^{d-1} = 1$, then
    $$1^{n(d-1) + 1} - \zeta^{n(d-1) + 1} = 1 - \zeta \neq 0.$$
    Thus $G$ has no roots in common with $x_1^{n(d-1)+1} - x_0^{n(d-1)+1}$.
\end{enumerate}
Since $p \nmid d-1$, the characteristic hypothesis for the $e = 0$ family in Proposition \ref{unified eg} is also met. We conclude that $J^*_F$ is irreducible over $\Bar\FF_p$, hence over $\Bar\QQ$.
\end{proof}


\begin{rmk}
For the polynomial $G$ above, it can be shown that
\[\Disc(G) = -(d-1)^{d-3} \text{ and } \Res(G, x_1^{n(d-1)+1} - x_0^{n(d-1)+1}) = (-1)^{n(d-1)} (d-1).\]
\end{rmk}

\section{Inseparability} \label{sec:insep}

Inseparable maps are a common source of counterexamples in algebraic geometry over algebraically closed fields $K$ of positive characteristic $p$. In this section, we interpret inseparability in terms of the critical locus and the flat Jacobian.

\subsection{Checking separability}

Recall that a generically finite morphism of integral schemes is \emph{(generically) separable} if it is dominant and the induced extension of function fields is
separable. We say a rational map of varieties is \emph{(generically) separable} if it is represented on an open subset by a generically separable morphism. Over fields of characteristic $0$, a rational map is generically separable if and only if it is dominant. In positive characteristic, separability can be delicate to check. The following well-known criterion relates separability to the critical locus.

\begin{propn} \label{prop_cf_and_sep} 
Let $f : X \dashrightarrow Y$ be a rational map of equidimensional smooth varieties over an algebraically closed field $K$.
Then the following are equivalent.
\begin{enumerate}
    \item The map $f$ is generically separable.
    \item The critical locus $C_f$ is a proper subset of $X \smallsetminus I_f$.
    \item Every irreducible component of $C_f$ has codimension $1$ in $X \smallsetminus I_f$.
\end{enumerate}
\end{propn}

\begin{proof}
Let $r := \dim X = \dim Y$. Replacing $X$ by $X \smallsetminus I_f$, we reduce to the case that $f: X \to Y$ is a morphism of smooth varieties of dimension $r$.

The implication (3) $\Rightarrow$ (2) is immediate. For the converse, recall that $C_f$ is defined (\ref{defn_cf_fulton_style}) as the zero scheme of the line bundle homomorphism $\det f'$. 
If $C_f$ is a proper subset of $X$ then $\det f'$ is not identically 0. Since the ideals defining $C_f$ are principal, (2) $\Rightarrow$ (3). 

Now we show $(1) \iff (2)$. We prove this by cases depending on dominance of $f$. 

Suppose that $f$ is not dominant. Then (1) is vacuously false by definition of separability, so we need only show that $C_f = X$. 
Because the image of $f'$ lies in the tangent bundle of the variety $\Bar{f(X)}$,
we have
\[C_f \supseteq \{x \in X : \dim T_{f(x)} (\Bar{f(X)}) < r\}.\] 
For all $y$ in the smooth locus $V$ of $\Bar{f(X)}$, we have $\dim T_y(\Bar{f(X)}) = \dim \Bar{f(X)}$;
and since $f$ is not dominant, $\dim \Bar{f(X)} < \dim Y = r$.
Thus $C_f$ contains the Zariski-dense open set $f^{-1}(V)$, 
so it is all of $X$.

Now suppose that $f$ is dominant. Then $f$ induces an extension $K(X)/K(Y)$ of function fields, and
$$f \text{ is generically separable } \iff K(X)/K(Y) \text{ is separable}.$$
Since $\dim X = \dim Y$, the extension $K(X)/K(Y)$ is finite algebraic, so by \cite[Theorem II.8.6A]{Hartshorne},
$$K(X)/K(Y) \text{ is separable } \iff \Omega_{K(X)/K(Y)} = 0.$$
In the First Exact Sequence for differentials \cite[II.8.3.A]{Hartshorne}
$$\Omega_{K(Y)/K} \otimes_K K(X) \to \Omega_{K(X)/K} \to \Omega_{K(X)/K(Y)} \to 0,$$
the first map is $(\delta f)_\eta : (f^* \Omega_{Y/K})_\eta \to (\Omega_{X/K})_\eta$, where $\eta$ is the generic point of $X$.
Thus
\begin{align*}
    \Omega_{K(X)/K(Y)} = 0 & \iff (\delta f)_\eta \text{ is surjective} & \text{(by exactness)} \\
    & \iff \det {(\delta f_\eta)} \neq 0 \\
    & \iff (\det \delta f)_\eta \neq 0 & \text{(by naturality)}\\
    & \iff \det f'_\eta \neq 0 & \text{(by duality)} \\
    & \iff \det f' \neq 0 \\
    & \iff C_f \subsetneq X & \text{(since $C_f = V(\det f')$}).
\end{align*}
Thus (1) is equivalent to (2).
\end{proof}

In dimension $n = 1$ and degree $d \geq 1$, an endomorphism $f \in \End^1_d$ is separable if and only if it ``factors through Frobenius'', i.e.\ $f = [F_0 : F_1]$ where $F_0, F_1 \in K[x_0^p, x_1^p]$ and $p = \Char K > 0$. 
In higher dimension, there is no such simple criterion; non--multiples-of-$p$ may appear as exponents (see Example \ref{eg_insep} below). One benefit of the flat Jacobian is that it can be used to quickly check for inseparability. 

\begin{cor} \label{cor_flat_J_and_not_separable}
Let $n, d \geq 1$ and $f = [F] \in \Rat^n_d$. Then $f$ is not separable if and only if $J^*_F = 0$.
\end{cor}
\begin{proof}
Apply Proposition \ref{prop_cf_and_sep} to $X = Y = \PP^n$ and use Theorem \ref{thm_flat_J_poly:b}. 
\end{proof}

Corollary \ref{cor_intro_insep} is the special case of the preceding when $f$ is an endomorphism, completing the proofs of formal statements made in the Introduction. 
We conclude this section with a geometric point of view and an example illustrating the simplicity of the method.

Let
\[
\Sep^n_d := \{ f \in \Rat^n_d : f \text{ is separable}\}
\]
and
\[
\SepEnd^n_d := \{ f \in \End^n_d : f \text{ is separable}\}
\]
The next Corollary shows how to view these sets as varieties, as we did for $\End_d^n$ and $\Rat_d^n$ in Lemma \ref{rat end qp}.

\begin{cor} \label{sep end qp}
For all $n \geq 1$ and $d \geq 1$, the sets
\[
\SepEnd^n_d \subseteq \Sep^n_d \subseteq \Bar\End^n_d
\]
are nonempty and Zariski-open.
\end{cor}

\begin{proof}
By Corollary \ref{cor_flat_J_and_not_separable},
we have $\Sep^n_d = \Rat^n_d{} \cap V(\Xi)$ where $\Xi$ is the homogeneous ideal generated by the coefficients $\Xi_\beta$ of $J^*_{n,d}$. 
Thus $\Sep^n_d$ is Zariski-open. Since $\End^n_d$ is also  Zariski-open, the intersection $\SepEnd^n_d$ is Zariski-open.
We now show $\SepEnd^n_d$ is nonempty. 
When $d = 1$, use the identity map. When $d > 1$, let
\begin{align*} 
    F_0 &:= x_0^d \\
    F_i &:= x_0^{d-1} x_i + x_i^d \qquad (1 \le i \le n).
\end{align*}
It is easy to see that $F$ is a morphism, 
and by Corollary \ref{Jstar formula}
\small 
\[J^*_F = \frac{1}{x_0} \begin{vmatrix}
    x_0^d \\ 
    x_0^{d-1} x_1 + x_1^d & x_0^{d-1} + dx_1^{d-1} \\
    \vdots & & \ddots \\
    x_0^{d-1} x_n + x_n^d & & & x_0^{d-1} + dx_n^{d-1}
\end{vmatrix}
= x_0^{d-1} \prod_{i=1}^n (x_0^{d-1} + dx_i^{d-1})\]
\normalsize 
is monic in $x_0$.
\end{proof}

\begin{eg} \label{eg_most_pathological}
If $\Char K \mid d$, then the tuple $F$ exhibited in the proof of Corollary \ref{sep end qp} defines a morphism $f$ whose critical scheme $C_f$ is a hyperplane of multiplicity $(n+1)(d-1)$, which in some sense is the most degenerate critical hypersurface possible.
For more on this phenomenon, see Ingram's work \cite{Ingram} on minimally critical maps in characteristic 0 and Faber's work \cite{Faber} on unicritical maps in dimension 1.
\end{eg}

\begin{eg} \label{eg_insep}
Let $K$ be an algebraically closed field of characteristic $3$, and consider the dominant rational map
$f : \PP^2 \dashrightarrow \PP^2$
defined by
\[F := (x^2 y, y^2 z, z^2 x).\]
Using the flat Jacobian, we can quickly check that $f$ is inseparable, despite no cubes appearing in $F$. 
By Corollary \ref{Jstar formula},
\[
J^*_F =
\frac{1}{x} \begin{vmatrix}
    x^2 y & x^2 & 0 \\
    y^2 z & 2 y z & y^2 \\
    z^2 x & 0 & 2 xz \\
\end{vmatrix} = 3 x^2 y^2 z^2
\]
vanishes in $K$. Thus by Corollary \ref{cor_flat_J_and_not_separable}, the map $f$ is not separable.
\end{eg}

For inseparable maps such as in Example \ref{eg_insep}, we have $J_F = J^*_F = 0$. So for these maps, two wrongs make a right---the homogeneous Jacobian determinant $J_F$ does indeed correctly compute the critical locus.

\subsection{The parameter scheme of separable endomorphisms} \label{sect_separable_scheme}

We conclude with a discussion of how to interpret these results ``over $\ZZ$'' instead of after choosing a base field. All our discussion in this section are just reformulations of the previous in different language.

In positive characteristic algebraic geometry and algebraic dynamics, inseparable maps are generally considered exceptional. One would prefer to work solely in a space of separable maps, to be able to speak of a general separable map, to construct moduli spaces of separable maps, to know the dimension of the set of inseparable maps, and so forth. All this requires knowing that there is a natural scheme structure on the set of separable maps. Our results on $J^*_{n,d}$ justify the existence of a parameter scheme $\SepEnd^n_{d/\ZZ}$ of separable maps, as we now explain.

We earlier introduced, for each algebraically closed field $K$, varieties $\Sep^n_d$ and $\End_d^n \subseteq \Bar\End_d^n$ that parametrize separable rational maps, all endomorphisms, and projective tuples of degree $d$ on $\PP^n$ over $K$, respectively. To emphasize the base field, we now denote these varieties $\Sep^n_{d/K}$ and $\End^n_{d/K} \subseteq \Bar\End^n_{d/K}$.
More generally, for any commutative unital ring $R$, we define the $R$-scheme
$$\Bar{\End}_{d/R}^n = \Proj R[\mathbf{u}]$$
where there is one indeterminate $u_{i, \alpha}$ for each $0 \leq i \leq n$ and each multiindex $\alpha$ of degree $d$ in $n + 1$ variables. Let $F$ be the universal tuple of degree $d$ in dimension $n$ over $R$, i.e. $F = (F_0, \ldots, F_n)$ where 
$$F_i = \sum_{\abs{\alpha} = d} u_{i, \alpha} x^\alpha \in R[\mathbf{u}][\mathbf{x}].$$
Let
$$\End_{d/R}^n = \Proj R[\mathbf{u}] \smallsetminus V(\Res(F_0, \ldots, F_n)).$$
By naturality of the resultant,
$$\End_{d/R}^n = \End_{d/\ZZ}^n \times_\ZZ \Spec R.$$
The bihomogeneity of $J^*_{n,d}$ allows us to define the $R$-scheme
$$C_{n,d/R} = V_{\End_{d/R}^n \times \PP^n_R}(\theta_R(J^*_{n,d}))$$
where $\theta_R$ is the specialization map of Section 4.
By naturality of $J^*_{n,d}$, 
$$C_{n,d/R} = C_{n,d/\ZZ} \times_\ZZ \Spec R.$$

Consider the first projection $\pi : C_{n,d/\ZZ} \to \End_{d/\ZZ}^n$. By general flatness, there is a maximal open subscheme $S^*$ of $\End_{d/\ZZ}^n$ such that $\pi|_{\pi^{-1}(S^*)}$ is flat. 
Since $J^*_{n,d} \neq 0$, the generic fiber of $\pi$ has codimension $1$. This implies that $S^*$ is the complement of the simultaneous vanishing locus of the coefficients $\Xi_\beta \in \ZZ[\mathbf{u}]$ of $J^*_{n,d}$.

So far we have not used any truly essential properties of $J^*_{n,d}$; the above claims hold just as well for $J_{n,d}$ or indeed for any nonzero bihomogenous polynomial in $\ZZ[\mathbf{u}][\mathbf{x}]$.
But in the case of $J^*_{n,d}$, the above claims have geometric interpretations.

First, by Theorem \ref{thm_flat_J_poly:b}, for every algebraically closed field $K$, the scheme $C_{n, d/ K}$ parametrizes degree-$d$ endomorphisms of $\PP^n_K$ with a marked critical point. Second, by Corollary \ref{cor_flat_J_and_not_separable}, the scheme $S^*$ over which $\pi$ is flat (when $n, d \geq 1$) has the property that for all algebraically closed fields $K$, the base change $S^* \times_\ZZ \Spec K$ is precisely $\SepEnd^n_{d/K}$.
In this sense, $\SepEnd^n_{d/\ZZ} := S^*$ is the parameter scheme of separable endomorphisms. 

In contrast, the flatness locus of $J_{n,d}$ is $\End_{d/\ZZ[1/d]}^n$. This is because fibers over primes dividing $d$ jump in dimension, and all geometric fibers over primes not dividing $d$ are critical hypersurfaces (Theorem \ref{thm_flat_J_poly}).
The scheme $\End_{d/\ZZ[1/d]}^n$ throws away all the inseparable maps at primes dividing $d$. We call $J^*_{n,d}$ the flat Jacobian because it avoids the jump in fiber dimension at these wild primes.

\bibliographystyle{abbrv}
\bibliography{critrefs}

\end{document}